\documentclass[10pt,twoside, a4paper,reqno]{amsart}

\usepackage{amscd,amssymb,amsmath,graphicx,verbatim,mathrsfs,xcolor}
\usepackage{wasysym}
\usepackage{hyperref}
\usepackage{enumitem}
\usepackage{microtype}
\usepackage[capitalize]{cleveref}
\usepackage{mathtools}
\usepackage{geometry}
\usepackage{multicol}

\makeatletter
\newcommand{\mylabel}[2]{#2\def\@currentlabel{#2}\label{#1}}
\makeatother

\mathchardef\mhyphen="2D 

\makeatletter
\newtheorem*{rep@theorem}{\rep@title}
\newcommand{\newreptheorem}[2]{%
	\newenvironment{rep#1}[1]{%
		\def\rep@title{#2 \ref{##1}}%
		\begin{rep@theorem}}%
		{\end{rep@theorem}}}
\makeatother
\newreptheorem{theorem}{Theorem}
\newreptheorem{lemma}{Lemma}

\usepackage{tikz,mathrsfs}
\usetikzlibrary{backgrounds}
\usetikzlibrary{arrows,cd}

\tikzset{>=stealth'}
\tikzset{
	symbol/.style={
		draw=none,
		every to/.append style={
			edge node={node [sloped, allow upside down, auto=false]{$#1$}}}
}}

\tikzcdset{scale cd/.style={every label/.append style={scale=#1},
		cells={nodes={scale=#1}}}}
\def\arrowLengthDisplayStyle{4ex}
\def\arrowHeightDisplayStyle{.8ex}
\def\arrowSkipDisplayStyle{.5ex}
\def\arrowLengthTextStyle{3ex}
\def\arrowHeightTextStyle{.8ex}
\def\arrowSkipTextStyle{.4ex}
\def\arrowLengthScriptStyle{2.5ex}
\def\arrowHeightScriptStyle{.6ex}
\def\arrowSkipScriptStyle{.3ex}
\def\arrowLengthScriptScriptStyle{2ex}
\def\arrowHeightScriptScriptStyle{.4ex}
\def\arrowSkipScriptScriptStyle{.2ex}

\renewcommand{\to}{\arrow{->}}
\newcommand{\mono}{\arrow{>->}}
\newcommand{\epi}{\arrow{->>}}
\newcommand{\embed}{\arrow{right hook->}}

\counterwithin*{equation}{section}

\newcommand{\MakeTikzArrowWithSuperscriptSubscript}[4]
{
	\mathchoice
	{ 
		\hspace*{\arrowSkipDisplayStyle}
		\begin{tikzpicture}[baseline]
		\draw [#1] (0,\arrowHeightDisplayStyle) -- node [above] {$#2$} node [below] {$#3$} (#4 * \arrowLengthDisplayStyle, \arrowHeightDisplayStyle);
		\end{tikzpicture}
		\hspace*{\arrowSkipDisplayStyle}
	}
	{ 
		\hspace*{\arrowSkipTextStyle}
		\begin{tikzpicture}[baseline]
		\draw [#1] (0,\arrowHeightTextStyle) -- node [above] {$\scriptstyle #2$} node [below] {$\scriptstyle #3$} (#4 * \arrowLengthTextStyle, \arrowHeightTextStyle);
		\end{tikzpicture}
		\hspace*{\arrowSkipTextStyle}
	}
	{ 
		\hspace*{\arrowSkipScriptStyle}
		\begin{tikzpicture}[baseline]
		\draw [#1] (0,\arrowHeightScriptStyle) -- node [above] {$\scriptscriptstyle #2$} node [below] {$\scriptscriptstyle #3$} (#4 * \arrowLengthScriptStyle, \arrowHeightScriptStyle);
		\end{tikzpicture}
		\hspace*{\arrowSkipScriptStyle}
	}
	{ 
		\hspace*{\arrowSkipScriptScriptStyle}
		\begin{tikzpicture}[baseline]
		\draw [#1] (0,\arrowHeightScriptScriptStyle) -- node [above] {$\scriptscriptstyle #2$} node [below] {$\scriptscriptstyle #3$} (#4 * \arrowLengthScriptScriptStyle, \arrowHeightScriptScriptStyle);
		\end{tikzpicture}
		\hspace*{\arrowSkipScriptScriptStyle}
	}
}

\newcommand{\MakeTikzArrowWithCentralLabel}[3]
{
	\mathchoice
	{ 
		\hspace*{\arrowSkipDisplayStyle}
		\begin{tikzpicture}[baseline]
		\draw [#1] (0,\arrowHeightDisplayStyle) -- node [fill=white,inner sep=1pt] {$#2$} (#3 * \arrowLengthDisplayStyle, \arrowHeightDisplayStyle);
		\end{tikzpicture}
		\hspace*{\arrowSkipDisplayStyle}
	}
	{ 
		\hspace*{\arrowSkipTextStyle}
		\begin{tikzpicture}[baseline]
		\draw [#1] (0,\arrowHeightTextStyle) -- node [fill=white,inner sep=1pt] {$\scriptstyle #2$} (#3 * \arrowLengthTextStyle, \arrowHeightTextStyle);
		\end{tikzpicture}
		\hspace*{\arrowSkipTextStyle}
	}
	{ 
		\hspace*{\arrowSkipScriptStyle}
		\begin{tikzpicture}[baseline]
		\draw [#1] (0,\arrowHeightScriptStyle) -- node [fill=white,inner sep=1pt] {$\scriptscriptstyle #2$} (#3 * \arrowLengthScriptStyle, \arrowHeightScriptStyle);
		\end{tikzpicture}
		\hspace*{\arrowSkipScriptStyle}
	}
	{ 
		\hspace*{\arrowSkipScriptScriptStyle}
		\begin{tikzpicture}[baseline]
		\draw [#1] (0,\arrowHeightScriptScriptStyle) -- node [fill=white,inner sep=1pt] {$\scriptscriptstyle #2$} (#3 * \arrowLengthScriptScriptStyle, \arrowHeightScriptScriptStyle);
		\end{tikzpicture}
		\hspace*{\arrowSkipScriptScriptStyle}
	}
}

\def\arrow#1{\def\lastArrowStyle{#1}
	\futurelet\testchar\arrowMaybeStreched}
\def\arrowMaybeStreched{\ifx[\testchar \let\next\arrowStreched
	\else \let\next\arrowUnstreched \fi
	\next}

\def\arrowStreched[#1]{\def\lastArrowStrech{#1}
	\futurelet\testchar\arrowMaybeLabel}
\def\arrowUnstreched{\def\lastArrowStrech{1}
	\futurelet\testchar\arrowMaybeLabel}

\def\arrowMaybeLabel{\ifx^\testchar \let\next\arrowSuperscript
	\else \ifx_\testchar \let\next\arrowSubscript
	\else \ifx~\testchar \let\next\arrowCentralLabel
	\else \let\next\arrowNoLabel
	\fi
	\fi
	\fi
	\next}

\def\arrowSuperscript^#1{\def\lastArrowSuperscript{#1}
	\futurelet\testchar\arrowSuperMaybeSub}
\def\arrowSuperMaybeSub{\ifx_\testchar \let\next\arrowSuperscriptSubscript
	\else \let\next\arrowSuperscriptNoSubscript \fi
	\next}

\def\arrowSubscript_#1{\def\lastArrowSubscript{#1}
	\futurelet\testchar\arrowSubMaybeSuper}
\def\arrowSubMaybeSuper{\ifx^\testchar \let\next\arrowSubscriptSuperscript
	\else \let\next\arrowSubscriptNoSuperscript \fi
	\next}

\def\arrowSuperscriptSubscript_#1{\def\lastArrowSubscript{#1}
	\arrowDrawSupSub}
\def\arrowSuperscriptNoSubscript{\def\lastArrowSubscript{}
	\arrowDrawSupSub}
\def\arrowSubscriptSuperscript^#1{\def\lastArrowSuperscript{#1}
	\arrowDrawSupSub}
\def\arrowSubscriptNoSuperscript{\def\lastArrowSuperscript{}
	\arrowDrawSupSub}
\def\arrowNoLabel{\def\lastArrowSuperscript{}
	\def\lastArrowSubscript{}
	\arrowDrawSupSub}

\def\arrowCentralLabel~#1{\MakeTikzArrowWithCentralLabel{\lastArrowStyle}{#1}{\lastArrowStrech}}
\def\arrowDrawSupSub{\MakeTikzArrowWithSuperscriptSubscript{\lastArrowStyle}{\lastArrowSuperscript}{\lastArrowSubscript}{\lastArrowStrech}}

\makeatletter
\renewcommand*\env@matrix[1][\arraystretch]{%
	\edef\arraystretch{#1}%
	\hskip -\arraycolsep
	\let\@ifnextchar\new@ifnextchar
	\array{*\c@MaxMatrixCols c}}
\makeatother

\tikzcdset{arrow style=tikz, diagrams={>=stealth'}, row sep=4em, column sep=4em}

\usepackage{amssymb,tikz,mathrsfs}

\renewcommand{\hat}{\widehat}
\newcommand{\A}{\mathscr{A}}
\newcommand{\B}{\mathscr{B}}
\newcommand{\C}{\mathscr{C}}
\newcommand{\D}{\mathscr{D}}

\renewcommand{\c}{\mathcal{C}}
\newcommand{\cof}{_{\dag}}
\newcommand{\fib}{^{\dag}}

\newcommand{\op}{^{\mathrm{op}}}
\newcommand{\Z}{\mathbb{Z}}
\DeclareMathOperator{\Ab}{Ab}
\DeclareMathOperator{\sSet}{sSet}

\DeclareMathOperator{\Ex}{Ex}

\DeclareMathOperator{\map}{map}

\DeclareMathOperator{\Ext}{Ext}
\DeclareMathOperator{\Extcat}{\mathscr{E}\!\mathit{xt}}

\DeclareMathOperator{\kan}{\mathcal{S}}
\DeclareMathOperator{\Sp}{\mathrm{Sp}}
\DeclareMathOperator{\homcat}{h\!}

\DeclareMathOperator{\seq}{\mathbb{S}}

\DeclareMathOperator{\Hst}{\mathcal{H}^{\mathrm{st}}}

\DeclareMathOperator{\cofibre}{cof}
\DeclareMathOperator{\Mod}{Mod}
\DeclareMathOperator{\classExtcat}{\mathcal{E}xt}

\newtheorem{theorem}{Theorem}[section]
\newtheorem{lemma}[theorem]{Lemma}
\newtheorem{corollary}[theorem]{Corollary}
\newtheorem{proposition}[theorem]{Proposition}

\theoremstyle{definition}
\newtheorem*{acknowledgements}{Acknowledgements}
\newtheorem{definition}[theorem]{Definition}
\newtheorem{remark}[theorem]{Remark}

\title[A theorem of Retakh and higher extension functors]{A theorem of Retakh for exact $\infty$-categories and higher extension functors}

\author[E. D. Børve]{Erlend D. Børve}
\address{Department of mathematical sciences, NTNU, NO-7491 Trondheim, Norway}
\curraddr{Institut Fourier, Université Grenoble Alpes, 100 rue des mathématiques, 38610 Gières, France}
\email{erlend.borve@univ-grenoble-alpes.fr}
\author[P. Trygsland]{Paul Trygsland}
\address{Department of mathematical sciences, NTNU, NO-7491 Trondheim, Norway}
\curraddr{Kjøpmannsgata 35, 7011 Trondheim, Norway}
\email{patrygsl@gmail.com}
\subjclass[2010]{14F35,18E10,18G30}

\date{}

\begin{document}

\begin{abstract}
	We define extension $\infty$-categories for exact $\infty$-categories in terms of bifibrations. Extension $\infty$-categories are invariant when passing to the stable hull, and consequently we show that they form an $\Omega$-spectrum, generalizing a theorem of Retakh. Finally, we show that the homotopy groups of extension $\infty$-categories are naturally isomorphic to the higher extension groups of the extriangulated category given by the homotopy category.
\end{abstract} 	
	
	\maketitle
	
	\tableofcontents

	\section{Introduction}
	
	Boardman--Vogt introduced the homotopy invariant structures known as $\infty$-categories in the early 1970s \cite{BV73}. Their motivation was to develop a machinery for better understanding infinite loop spaces, or ``homotopy abelian groups.'' A central idea is to avoid choices for compositions, but rather have a space encoding all possible choices. As an example, a loop space does not carry a natural well-defined concatenation of loops before applying~$\pi_0$, though it is weakly homotopy equivalent to a topological group~\cite[p.~31]{Ada78}. Nowadays, it is well-known that $\infty$-categories serve as a model for~$(\infty,1)$--categories~\cite{Joy02, Lur09}.
	
	Roughly speaking, an exact category~$\c$ is an additive category together with a collection of short exact sequences subject to certain constraints~\cite{Hel58}, (Keller's axioms \cite{Kel90} can be taken as a modern approach). This allows for homological algebra to be performed in a more general context than abelian categories. Non-abelian examples include the category of vector bundles of a scheme and the category of Banach spaces. In addition to homological algebra, exact categories provide a natural framework for~$\mathrm{K}$-theory. Quillen introduced the higher algebraic~$\mathrm{K}$--groups in this context~\cite{Qui73}.
	
	Barwick defines exact $\infty$-categories in order to generalise definitions and results in~$\mathrm{K}$-theory~\cite{Bar15, Bar16, BR}. This broadens the scope considerably; it captures both (nerves of) exact categories and stable $\infty$-categories. Moreover, every extension closed subcategory of a stable $\infty$-category has the structure of an exact $\infty$-category. In fact, all exact $\infty$-categories occur in this manner; Klemenc shows that all exact $\infty$-categories can be embedded into a stable hull \cite{Kle20}. One would thus expect more constructions and results about exact categories to generalise, paving the way for potentially useful applications. In particular, one could study homological algebra from an $\infty$-categorical point of view.

	To better understand the homological algebra of exact $\infty$-categories~$\C$, we define $\infty$-categorical analogues of extension categories, denoted by~$\Extcat^n_{\C}(b,a)$, where $a$ and $b$ are objects in $\C$. These generalise the well-established extension categories for exact categories. It is well-known that the Yoneda Ext-groups~$\Ext^n_{\c}(b,a)$ of an exact category $\c$ can be recovered as the zeroth homotopy groups of~$\Extcat^n_{\c}(b,a)$ \cite[Proposition~XII.4.4]{Mac67}.
	
	To modernise the notion of extension ($\infty$-)category, we use the language of bifibrations (\Cref{def:bifib}). 
	For an exact $\infty$-category~$\C$, we declare that~$\Extcat^0_{\C}(b,a)\coloneqq \hom_{\C}(b,a)$, the mapping space from~$b$ to~$a$. More precisely, we define $\Extcat^0_{\C}$ to be the mapping space bifibration of $\C$. The first extension bifibration $\Extcat^1_{\C}$ is defined analogously to its $1$-categorical counterpart, and the higher one are obtained by compositions of $\Extcat^1_{\C}$ with itself. 
	
	The main result of ours is the following.
	
	\begin{reptheorem}{thm:SpExtbifun}
		Let $\C$ be an exact $\infty$-category and let $\Sp$ be the $\infty$-category of $\Omega$-spectrum objects. There exists a functor
		\begin{equation}
		\Extcat_{\C}\colon\C\op\times\C\to \Sp
		\end{equation}
		such that the $n$-th component of the spectrum $\Extcat_{\C}(b,a)$ is weakly equivalent to the extension $\infty$-category $\Extcat^n_{\C}(b,a)$ for all $a,b\in \C$.
	\end{reptheorem}
	We deduce that the spectrum
	\begin{equation*}
	\Extcat_\C(b,a) \coloneqq \{\Extcat^n_{\C}(b,a)\}_{n\geq 0}
	\end{equation*}
	is an~$\Omega$-spectrum. This is a generalisation of a theorem of Retakh, who proved the same result for abelian categories \cite[Theorem~2(b)]{Ret86}. It follows that we have isomorphisms 
	\begin{equation}\label{eq:hermanniso}
	\begin{tikzcd}
	\pi_{-n}\Extcat_{\C}(b,a)\simeq \pi_{0}\Extcat^n_{\C}(b,a) \arrow[r]& \pi_{1}\Extcat^{n+1}_{\C}(b,a)
	\end{tikzcd}
	\end{equation}
	of homotopy groups. This means, in particular, that the higher structure of~$\Extcat_\C (b,a)$ descends to the abelian group structure on classical $\Ext$-groups.
	
	To prove Theorem~\ref{thm:SpExtbifun}, we show that the extension categories $\Extcat^n_{\C}(b,a)$ are weakly homotopy equivalent to $\map_{\mathcal{H}^{\mathrm{st}}_{\geq 0}(\C)}(b,\Sigma^n a)$, where $\mathcal{H}^{\mathrm{st}}_{\geq 0}(\C)$ is a certain subcategory of the stable hull of $\C$. The Retakh-esque spectra of extension $\infty$-categories will be weakly equivalent to Lurie's mapping spectra in the prestable $\infty$-category $\mathcal{H}^{\mathrm{st}}_{\geq 0}(\C)$.
	
	There is a notion extriangulated categories, defined by Nakaoka--Palu \cite{NP19}, which simultaneously generalises exact and triangulated categories. Recently, Nakaoka--Palu showed that the homotopy category of an exact $\infty$-category has a natural extriangulated structure \cite[Theorem~4.22]{NP20}. In particular, it encapsulates the now well-known result that the homotopy category of a stable $\infty$-category is canonically triangulated \cite[Theorem~1.1.2.14]{Lur17}, and moreover the homotopy category of the nerve of an exact category is obviously exact. 

	For a non-negative integer~$n$, Gorsky--Nakaoka--Palu \cite[Definition~3.1]{GNP21} define the~$n$-extension groups for extriangulated categories, ultimately inspired by Yoneda's classic monograph on extension categories \cite{Yon60}. We show that the higher structure contained in the Retakh spectra~$\Extcat_{\C}(-,-)$ descends to the~$n$-extension groups in the homotopy category.
	
	\begin{reptheorem}{thm:higherext}
		Let~$\C$ be an exact $\infty$-category. Then the bifunctors $\pi_0\Extcat^n_{\C}(-,-)$ on the extriangulated category $(\homcat\C,\pi_0\Extcat^1_{\C}(-,-))$ is naturally isomorphic to the $n$-th extension functor of~$\homcat \C$.
	\end{reptheorem}

	\textbf{Outline.} Section~\ref{sec:exqcat} contains little to no original ideas. It is concerned with basics of exact $\infty$-categories and diagram lemmas. We define extension categories for $\infty$-categories in Section~\ref{section:extensioncategories}. \Cref{thm:SpExtbifun} is stated and proved in Section \ref{sec:retakh}. In the final Section~\ref{section:extriangulated}, we discuss how Retakh spectra determine the higher structure of the extriangulation on homotopy categories. In particular, Theorem~\ref{thm:higherext} is proven.

	\textbf{Notation and conventions.} Throughout, we fix a Grothendieck universe~$U$. Simplicial sets should thus be understood as simplicial~$U$-sets. We refer to Shulman \cite[Section~8]{Shu08} for a detailed treatment of Grothendieck universes. 
	
	A category~$J$ will be identified with its nerve~$\mathrm{N}(J)$, so that it can be regarded as an $\infty$-category. An~$\infty$-functor between $\infty$-categories~$\C$ and~$\D$ (i.e. a morphism of the underlying simplicial sets) will simply be referred to as a functor.
	
	The~$0$-simplices of an $\infty$-category~$\C$ are often referred to as objects. Similarly,~$1$-simplices are referred to as maps or morphisms.
	We write~$f\sim g$ when the 1-simplicies~$f$ and~$g$ are homotopic. A degenerate 1-simplex~$s_0X$, or a map which is homotopic to it, is denoted by \begin{tikzcd}X\arrow[r,equal]& X\end{tikzcd}.
	
	For two simplicial sets~$X$ and~$Y$, we denote  function complex (or internal hom) by~$\map(X,Y)$, whose set of~$n$-simplices is~$\map(X,Y)_n={\sSet}(X\times [n], Y)$. Note that this space models the homotopy function complex \cite{DK80}.
	In the case of~$\map(\Delta^1,\C)$, where~$\C$ is an $\infty$-category, there is the subcomplex~$\map_{\C}(b,a)$ whose~$0$-simplices are the~$1$-simplices~$b\to a$ in~$\C$. 
	For an ordinary category~$\c$ we denote by~$\hom_{\c}(b,a)$ the set of maps~$b\rightarrow a$.
	
	A \textit{subcategory} of an $\infty$-category $\C$ is a simplicial subset $\mathscr{D}\subseteq \C$ such that the square
	\begin{equation*}
	\begin{tikzcd}
	\mathscr{D}\arrow[d] \arrow[r,hook]& \C\arrow[d] \\
	\mathrm{N}(h\mathscr{D}) \arrow[r,hook]& \mathrm{N}(h\C)
	\end{tikzcd}
	\end{equation*}
	is a pullback diagram in the category of simplicial sets. Alternatively, one can impose the inclusion to be an inner fibration \cite[\href{https://kerodon.net/tag/01CF}{Tag 01CF}]{Ker}. In any case, it is uniquely specified by the subcategory $ h\mathscr{D}$ of $h\C$.
	
	A \textit{homotopy coherent diagram} of~$J$ in an $\infty$-category~$\C$ is a functor~$J\to \C$. A \textit{homotopy commutative diagram} is a functor from~$J$ into the homotopy category~$\homcat\C$.
	
	Recall that the limit (resp. colimit) of a diagram~$D\colon J\to \C$ arises as a left (resp. right) adjoint of the diagonal functor~$\C\to \map(J,\C)$.
	
	A \textit{biproduct} of a family $\{X_i\}_{i\in I}$ of objects in $\C$ is both a product and a coproduct of this family. If it exists, it will be denoted by~$\bigoplus\limits_{i\in I} X_i$.
	
	Let~$\bigoplus\limits_{i=1}^n X_i$ and~$\bigoplus\limits_{j=1}^m Y_j$ be biproducts in~$\C$. Then a map~$f\colon \bigoplus\limits_{i=1}^n X_i\to \bigoplus\limits_{j=1}^m Y_j$ is uniquely determined, up to homotopy, by its components~$f_{j,i}\colon X_i\to Y_j$. We will thus write~$f$ as a matrix~$(f_{j,i})_{i,j}$.
	
	\begin{acknowledgements}
		We are grateful to Rune Haugseng and Yann Palu for helpful comments.
	\end{acknowledgements}
	
	\section{Exact $\infty$-categories}\label{sec:exqcat}
	
	We first review exact $\infty$-categories, a generalisation of exact categories to the realm of $\infty$-categories.
	Since the underlying category of an exact category is additive, it is only sensible to introduce the $\infty$-categorical notion.
	\begin{definition}[{\cite[§2]{Bar15}}]\label{addcat}
		An $\infty$-category~$\C$ is \emph{additive} if the following hold.
		\begin{enumerate}[label=(Add\arabic*), leftmargin=1.3cm]
			\item\label{addcat1} There is a zero object 0 in~$\C$, which is to say that~$\map_\C (0,x)$ and~$\map_\C (x,0)$ are contractible for all~$x\in\C$.
			\item\label{addcat2} Finite products and coproducts exist in~$\C$.
			\item\label{addcat3} For any finite family of objects~$\{x_1,\dots, x_n\}$, the natural map 
			\begin{equation*}
			\coprod_{i=1}^n x_i \to \prod_{i=1}^n x_i,
			\end{equation*}
			induced by the identity maps~$x_i\to^{1}x_i$, is a homotopy equivalence.
			\item For all~$x,y\in\C_0$, the Hom-set~$\hom_{h\C}(x,y)\coloneqq \pi_0\map_{\C}(x,y)$ admits an abelian group structure, where we define ~$f+g$ to be the composite
			\begin{center}
				\begin{tikzcd}[ampersand replacement=\&]
					x\arrow[r,"{\begin{pmatrix} 1 \\ 1 \end{pmatrix}}"] \& x\coprod x\arrow[r,"{\begin{pmatrix} f & 0 \\ 0 & g \end{pmatrix}}"] \& y\prod y \arrow[r,"{\begin{pmatrix} 1 & 1 \end{pmatrix}}"] \& y
				\end{tikzcd}
			\end{center}
			in~$\homcat \C$.
		\end{enumerate}
	\end{definition}
	A bicomplete $\infty$-category~$\C$ is additive precisely when the homotopy category~$\homcat \C$ is additive as an ordinary category. Analogously to how  biproducts arise in additive categories, the axioms \ref{addcat2} and \ref{addcat3} imply that~$\C$ has finite biproducts.
	
	Barwick defines exact $\infty$-categories \cite[Definition~3.1]{Bar15} by adapting Keller's minimal set of axioms \cite{Kel90} to the $\infty$-categorical setting.
	\begin{definition}
		Let~$\C$ be an additive $\infty$-category and let~$\C\cof$ and~$\C\fib$ be subcategories that contain all objects in~$\C$, as well as all homotopy equivalences. The maps in~$\C\cof$ will be referred to as \emph{cofibrations}, whereas morphisms in~$\C\fib$ are \emph{fibrations}. The triple~$(\C,\C\cof,\C\fib)$ is called an \textit{exact $\infty$-category} if the following axioms hold.
		\begin{enumerate}[label=(Ex\arabic*), leftmargin=1.1cm]
			\item\label{Ex1} For any zero object~$0$ in~$\C_0$, all morphisms of the form~$0\to x$ are cofibrations and those of the form~$x\to 0$ are fibrations.  
			\item\label{Ex2} Pushouts of cofibrations exist and are cofibrations, and dually for pullbacks of fibrations.
			\item\label{Ex3} The following are equivalent for a homotopy coherent square 
			\begin{center}
				\begin{tikzcd}
					a\arrow[r,"i"]\arrow[d,"f"] & e\arrow[d,"p"] \\
					c\arrow[r,"g"] & b
				\end{tikzcd}
			\end{center}
			\begin{enumerate}[label=(Ex3.\arabic*), leftmargin=1.4cm]
				\item\label{Ex3a} The square is a pullback, the map~$g$ is a cofibration and~$p$ is a fibration.
				\item\label{Ex3b} The square is a pushout, the map~$i$ is a cofibration and~$f$ is a fibration.
			\end{enumerate}
		\end{enumerate}
	\end{definition}
	If the subcategories~$\C\cof$ and~$\C\fib$ are implicitly specified, we simply say that~$\C$ is exact. A triple~$(\C,\C\cof,\C\fib)$ satisfying the axioms above is called an \textit{exact structure} on~$\C$. Whenever a cofibration appears in a diagram, it will be drawn as follows:~$\mono$ (like a monomorphism in an ordinary category). A fibration will be drawn as a two-headed arrow~$\epi$.
	
	Nerves of exact categories are exact $\infty$-categories, where the class of cofibrations consists of the admissible monomorphisms, and the fibrations are the admissible epimorphisms. Moreover, any stable $\infty$-category \cite[Definition~1.1.1.9]{Lur17} can be seen as an exact $\infty$-category where~$\C\cof=\C=\C\fib$. 
	
	The axioms above provide a framework for exact sequences, in more or less the usual fashion.
	\begin{definition}
		Let~$\C$ be an exact $\infty$-category. An \emph{exact sequence} in~$\C$ is a square 
		\begin{center}
			\begin{tikzcd}
				a\arrow[r]\arrow[d, two heads] & e\arrow[d] \\
				0\arrow[r, tail] & b
			\end{tikzcd}
		\end{center}
		where~$0$ is a zero object and the equivalent criteria in \ref{Ex3} are met.
	\end{definition}
	Equivalently, the exact sequences are bicartesian squares of the form
	\begin{center}
		\begin{tikzcd}
			a\arrow[r,"i",tail]\arrow[d,"",two heads]\arrow[dr, phantom, "\square"] & e\arrow[d,"p",two heads] \\
			0\arrow[r,"",tail] & b
		\end{tikzcd}
	\end{center}
	As is conventional for exact 1-categories, exact sequences can be drawn horizontally
	\begin{equation*}
	\begin{tikzcd}
	a\arrow[r,"i",tail]&e \arrow[r,"p",two heads]&b,
	\end{tikzcd}
	\end{equation*}
	omitting the zero object. There is not really any loss of information in such notation, as the choice of morphisms to/from zero is irrelevant up to homotopy.
	
	\begin{definition}\label{def:mapofseq}
		A \textit{map of exact sequences} is simply a map of bicartesian squares. 
		Such maps will mostly be depicted as homotopy commutative diagrams of the form
		\begin{equation*}
		\begin{tikzcd}
		a\arrow[r,"i",tail]\arrow[d]&e\arrow[r,"p",two heads]\arrow[d]&b\arrow[d] \\
		a'\arrow[r,"i'",tail]&e'\arrow[r,"p'",two heads]&b'
		\end{tikzcd}
		\end{equation*}
	\end{definition}

	Exact functors also have a completely analogous definition.
	\begin{definition}\label{def:exfunc}
		Let~$\C$ and~$\mathscr{D}$ be exact $\infty$-categories. A functor~$F\colon \C\to \mathscr{D}$ is \textit{exact} if it preserves zero objects, cofibrations and pushouts of cofibrations.
	\end{definition}
	When checking that a functor is exact, one may equivalently consider the dual statement for fibrations \cite[Proposition~4.8]{Bar15}. Exact sequences are obviously preserved by exact functors, and by considering split exact sequences one proves that they preserve the additive structure as well.
	
	Full subcategories~$\mathscr{D}$ of a given exact $\infty$-category~$\C$ inherit the exact structure if they are \textit{closed under extensions}, i.e. for all exact sequences 
	\begin{center}
		\begin{tikzcd}
			a\arrow[r,tail] & e\arrow[r,two heads] & b
		\end{tikzcd}
	\end{center}
	with~$a,b\in\mathscr{D}$, we have that~$e\in\mathscr{D}$. The class of cofibrations in $\mathscr{D}$ are the maps in~$\C\cof\cap \mathscr{D}$ for which the cofibre is in~$\mathscr{D}$, and dually for fibrations. The inclusion functor
	$
	\mathscr{D}\embed \C
	$
	is then exact.
	
	It turns out that all exact $\infty$-categories occur as extension closed subcategories of stable $\infty$-categories. 
	\begin{theorem}[{\cite[Theorem~1.2]{Kle20}}]\label{thm:Kle20.1}
		Given a small exact $\infty$-category~$\C$, there exists a stable hull~$\mathcal{H}^{\rm st}(\C)$ into which~$\C$ embeds exactly and universally. 
	\end{theorem}
	
	We will reveal some of the structure of $\mathcal{H}^{\rm st}(\C)$ when we need it in Section \ref{sec:retakh}. \Cref{thm:Kle20.1} can be seen as an $\infty$-categorical version of the Gabriel--Quillen embedding theorem for exact categories, which states that they embed exactly into abelian categories.
	
	A number of results concerning exact $\infty$-categories generalise to exact $\infty$-categories. We will make frequent use of our next lemma, first generalised by Barwick.
	\begin{lemma}[{\cite[Lemma~4.5]{Bar15}}]\label{Bar15.4.5}
		A homotopy coherent square
		\begin{center} 
			\begin{tikzcd}
				a\arrow[r,"i"]\arrow[d,"f"] & e\arrow[d,"p"] \\
				c\arrow[r,"g"] & b
			\end{tikzcd}
		\end{center}
		in an exact $\infty$-category is bicartesian if either of the following conditions holds.
		\begin{enumerate}
			\item\label{Bar15.4.5.1} The square is a pushout and~$i$ is a cofibration.
			\item\label{Bar15.4.5.2} The square is a pullback and~$p$ is a fibration. 
		\end{enumerate}
	\end{lemma}
	
	Generalisations of celebrated diagram lemmas will also be helpful in later sections. The first relates pushouts to maps of exact sequences. 
	
	\begin{lemma}[{\cite[Proposition~A.1]{Kle20}}]\label{Buh2.12}
		A pushout
		\begin{center}
			\begin{tikzcd}
				a \arrow[r, "i",tail] \arrow[d] \arrow[dr, phantom, "\ulcorner",very near end]  & e \arrow[d] \\
				
				c \arrow[r,tail]               & f                                    
			\end{tikzcd}
		\end{center}
		where~$i$ is cofibration, can be extended to a map of exact sequences 
		\begin{center}
			\begin{tikzcd}
				a \arrow[r, "i",tail] \arrow[d] \arrow[dr, phantom, "\ulcorner",very near end]  & e \arrow[d]\arrow[r,two heads,"p"] & b\arrow[d,equal] \\
				
				c \arrow[r,tail]               & f \arrow[r,two heads] & b                                  
			\end{tikzcd}
		\end{center}
		Conversely, the existence of such a map of exact sequences implies that the square is a pushout. A dual statement holds for pullbacks along fibrations.
	\end{lemma}

	Secondly, a map of exact sequences has a canonical factorisation.        
	\begin{lemma}[{\cite[Proposition~A.2]{Kle20}}]\label{Buh3.1}
		Any map of exact sequences
		\begin{center}
			\begin{tikzcd}
				a \arrow[r, "i_a",tail] \arrow[d,"f'"]   & e \arrow[d,"f"]\arrow[r,two heads,"p_b"] & b\arrow[d,"f''"] \\
				
				c \arrow[r,tail,"i_c"]               & f \arrow[r,two heads,"p_d"] & d                                  
			\end{tikzcd}
		\end{center}
		can be factored
		\begin{center}
			\begin{tikzcd}
				a \arrow[r, "i_a",tail] \arrow[d,"f'"]\arrow[dr, phantom, "\square"]   & e \arrow[d,"g"]\arrow[r,two heads,"p_b"] & b\arrow[d,equal] \\
				c \arrow[r, "j",tail] \arrow[d,equal]   & z \arrow[d,"h"]\arrow[r,two heads,"q"]\arrow[dr, phantom, "\square"] & b\arrow[d,"f''"] \\
				c \arrow[r,tail,"i_c"]               & f \arrow[r,two heads,"p_d"] & d                                     
			\end{tikzcd}
		\end{center}
		where the squares marked by~$\square$ are bicartesian.
	\end{lemma}
	
	The renowned Five Lemma is a consequence of \Cref{Buh3.1}. It will be proven in \Cref{app:diaproofs}, since there is no significant difference from the special case of exact 1-categories.
	\begin{lemma}[Five lemma]\label{5lem}
		Consider a map of exact sequences
		\begin{center}
			\begin{tikzcd}
				a \arrow[r, "i_a",tail] \arrow[d,"f'"]   & e \arrow[d,"f"]\arrow[r,two heads,"p_b"] & b\arrow[d,"f''"] \\
				
				c \arrow[r,tail,"i_c"]               & f \arrow[r,two heads,"p_d"] & d                                  
			\end{tikzcd}
		\end{center}
		If~$f'$ and~$f''$ are homotopy equivalences (resp. cofibrations, resp. fibrations), so is~$f$.  
	\end{lemma}
	
	As a curiosity, we also lift the $3\times 3$-lemma to the $\infty$-categorical setting. This proof is also deferred to \Cref{app:diaproofs}.
	\begin{lemma}[$3\times 3$--lemma]\label{lemma:3x3}
		Consider the homotopy coherent diagram with exact columns
		\begin{center}
			\begin{tikzcd}
				a' \arrow[d,tail,"i_a"]\arrow[r,"f'"] & e' \arrow[r,"g'"]\arrow[d,tail,"i_e"] & b'\arrow[d,tail,"i_b"] \\
				a \arrow[r,"f"]\arrow[d,two heads,"p_a"] & e \arrow[r,"g"]\arrow[d,two heads,"p_e"] & b\arrow[d,two heads,"p_b"] \\
				a''\arrow[r,"f''"]& e''\arrow[r,"g''"] & b''
			\end{tikzcd}
		\end{center}
		If the middle row and one of the other rows are exact, then the remaining row is exact.
	\end{lemma}
	
	\section{$\infty$-categories of extensions}
	\label{section:extensioncategories}
	Fix an exact $\infty$-category~$\C$. We will generalise MacLane's notion of extension category \cite{Mac67} to the $\infty$-categorical setting using modern techniques. Specifically, our constructions are defined in terms of bifibrations \cite[§2.4.7]{Lur09}. Since bifibrations specify bifunctors, we will arrive at a definition of extension bifunctors for exact $\infty$-categories.
	
	Let us first recall the definition of a bifibration.
	
	\begin{definition}[{\cite[Definition~2.4.7.2]{Lur09}}]\label{def:bifib}
		A map of simplicial sets $e\colon X\to S\times T$ is called a \textit{bifibration} provided that it is an satisfies the following criterion: for all positive integers $m$, given the solid part of the diagram
	 	\begin{equation*}
		\begin{tikzcd}
		\Lambda_i^m\arrow[d,hook] \arrow[r] & X\arrow[d,"e"]\\
		\Delta^m \arrow[ru,dashed]\arrow[r,"s"] & S\times T
		\end{tikzcd}
		\end{equation*}
		the dashed morphism renders the entire diagram commutative provided that
		\begin{enumerate}
			\setcounter{enumi}{-1}
			\item $0<i<m$, i.e. the $e$ is an inner fibration,
			\item $i=0$ and $\pi_T\circ s$ maps $\Delta^{\{0,1\}}\subset \Delta^n$ to a degenerate 1-simplex in $T$,
			\item $i=m$ and $\pi_S\circ s$ maps $\Delta^{\{m-1,m\}}\subset \Delta^m$ to a degenerate 1-simplex in $S$.
		\end{enumerate}
	\end{definition}

The prototypical example of a bifibration is the map-bifibration. For an $\infty$-category $\C$, it is given by
\begin{equation*}
\mathrm{eval}_0\times\mathrm{eval}_1\colon \map(\Delta^1,\C)\to \C\times \C
\end{equation*}
mapping an arrow to its source and target \cite[Corollary~2.4.7.11]{Lur09}. If $\C$ is exact, we will think of a $0$-extension in $\C$ as a map in $\C$. Henceforth, we denote $\Extcat^0(\C,\C)\coloneqq \map(\Delta^1,\C)$, and the mapping space bifibration by
\begin{equation*}
\begin{tikzcd}
	\Extcat^0(\C,\C) \arrow[r,"\epsilon_0"]& \C\times \C
\end{tikzcd}
\end{equation*}
Fixing two objects $a,b\in \C$, the fibre $\varepsilon_0^{-1}(b,a)$ is then the Kan complex of maps from $b$ to $a$. In general, given a bifibration $e\colon \to S\times T$, the fibres $e^{-1}(b,a)$ are always Kan complexes. Better still, any such bifibration determines a bifunctor
\begin{equation*}
e\colon \C\op\times \C \to \kan
\end{equation*}
into the $\infty$-category of Kan complexes \cite[Lemma 5.1.2]{AF20}. In particular, the mapping space bifunctor
\begin{equation*}
\map\colon \C\op\times\C\to \kan
\end{equation*}
is determined by $\varepsilon_0$.
	
Next, we define a bifibration yielding 1-extensions. Recall that~$[1]\times [1]$ consists of two~$2$-simplices glued along their~$1$-faces to obtain a square. A short exact sequence is thus a homotopy coherent diagram~$\mathbb{E}\colon [1]\times [1]\rightarrow \C$ which defines a bicartesian square
	\begin{center}
		\begin{tikzcd}
			a\arrow[r,"i",tail]\arrow[d,"",two heads]\arrow[dr, phantom, "\square"] & e\arrow[d,"p",two heads] \\
			0\arrow[r,"",tail] & b
		\end{tikzcd}
	\end{center}
The corners in~$[1]\times [1]$ are commonly labeled~$(k,l)$, where ~$k,l=0,1$. Define the $\infty$-category $\Extcat_{\C}^1(\C,\C)$ as the full subcategory of $\map(\Delta^1\times\Delta^1,\C)$ spanned by the short exact sequences. Every inclusion \\ $i_{l,k}\colon [0]\to {[1]}\times [1]$ induces a map~$\mathrm{eval}_{k,l}\colon \map([1]\times[1],\C)\to \C$ which restricts to~\[\mathrm{eval}_{k,l}\colon \Extcat^1_{\C}(\C,\C)\to \C.\] In particular, a short exact sequence
	\begin{center}
		\begin{tikzcd}
			\mathbb{E}\colon &  a \arrow[r, tail] & e \arrow[r, two heads] & b                           
		\end{tikzcd}
	\end{center}
	satisfies~$\mathrm{eval}_{0,0}\mathbb{E}=a$,~$\mathrm{eval}_{0,1}\mathbb{E}=0$,~$\mathrm{eval}_{1,0}\mathbb{E}=e$ and~$\mathrm{eval}_{1,1}\mathbb{E}=b$. Now define
	\begin{equation*}
	\epsilon_1\coloneqq \mathrm{eval}_{0,0}\times\mathrm{eval}_{1,1}\colon \Extcat_{\C}^1(\C,\C) \to \C\times\C.
	\end{equation*}
	Our next result, where we show that $\epsilon_1$ is a bifibration, implies that~$\Extcat^1_{\C}(b,a)$ does in fact define a bifunctor
	\[
	\Extcat^1_\C(-,-)\coloneqq \epsilon_1^{-1}\colon \C^{\mathrm{op}}\times \C \to \kan
	\]
	into the $\infty$-category of Kan complexes.
	
	\begin{proposition}\label{prop:e1bifib}
		\label{proposition:bifibration}
		The map~$\epsilon_1\colon \Extcat^1_{\C}(\C,\C)\to  \C\times \C$ is a bifibration.
	\end{proposition}
	
	\begin{proof}
	We first show that $\epsilon_1$ is an inner fibration. As $\Extcat_{\C}^1(\C,\C)$ is a subcategory of $\map(\Delta^1\times\Delta^1,\C)$, the inclusion 
	\begin{equation*}
	\Extcat_{\C}^1(\C,\C)\embed \map(\Delta^1\times\Delta^1,\C)
	\end{equation*}
	is an inner fibration by definition.
	Moreover, if we identify $\C\times\C$ with $\map(\ast \sqcup \ast,\C)$,
	we have an inner fibration \cite[\href{https://kerodon.net/tag/01BU}{Tag 01BU}]{Ker}
	\begin{equation*}
	\map(\Delta^1\times\Delta^1,\C) \to \map(\ast \sqcup \ast,\C) \cong \C\times\C.
	\end{equation*}
	A composite of inner fibrations is still an inner fibration \cite[\href{https://kerodon.net/tag/01BH}{Tag 01BH}]{Ker}, whence we can complete the first step of the proof by recalling that $\epsilon_1$ is the composite of the two inner fibrations displayed above.
	
	Suppose that the solid part of the following diagram commutes
		\begin{center}
			\begin{tikzcd}
				{\Lambda^m_0} \arrow[r, "\gamma"] \arrow[d] & \Extcat^1_{\C}(\C,\C) \arrow[d,"\epsilon_1"'] \\
				{\Delta^m}\arrow[r, "s"]\arrow[ru,dashed] & \C\times \C
			\end{tikzcd}
		\end{center}
		where the first~$1$-simplex in~$s$ has a degenerate~$1$-simplex as its second component. For $\epsilon_1$ to be a bifibration, the dashed morphism should exist, rendering the whole diagram commutative.
		There is also a dual horn filling requirement for~$\Lambda^m_m$ which is proved by dualising the reasoning in the next paragraphs.
		
		The case~$m=1$ reduces to completing a diagram
		\begin{center}
			\begin{tikzcd}
				a_0 \arrow[r, tail,"i"] \arrow[d,"f"]& e \arrow[r, two heads] & b\arrow[d, equal] \\
				a_1 & & b
			\end{tikzcd}
		\end{center}
		to a map of exact sequences (Definition~\ref{def:mapofseq}). We need only compute the pushout of $i$ along $f$ and invoke Lemma~\ref{Buh2.12}. The 1-simplex in $\Extcat_{\C}^1(\C,\C)$ given by the dashed morphism will be
		\begin{center}
			\begin{tikzcd}
				a_0 \arrow[r, tail,"i"] \arrow[d,"f"]\arrow[dr, phantom, "\square"]& e \arrow[r, two heads] \arrow[d]& b\arrow[d, equal] \\
				a_1\arrow[r, tail] & p \arrow[r, two heads] & b
			\end{tikzcd}
		\end{center}
		
		The case~$n\geq 2$ is also a consequence of Lemma~\ref{Buh2.12}. Indeed, the first~$1$-simplex in~$\gamma$ is of the form
		\begin{center}
			\begin{tikzcd}
				a_0 \arrow[r, tail] \arrow[d]& e_0 \arrow[r, two heads] \arrow[d]& b_0\arrow[d, equal] \\
				a_1\arrow[r, tail] & e_1 \arrow[r, two heads] & b_0
			\end{tikzcd}
		\end{center}
		whence the leftmost square is a pushout square due to Lemma~\ref{Buh2.12}. A lift of~$s$ to~$\Extcat^1_{\C}(\C,\C)$ can now be constructed using the fact that pushout squares are initial in the category of cocones.
	\end{proof}

	
	

	\begin{remark}
	For each pair of objects $a,b\in \C$, we deduce from \Cref{prop:e1bifib} that $\Extcat_{\C}^1(b,a)$ is a Kan complex. One can also prove this assertion using the Five Lemma (\Cref{5lem}). Indeed, the Five Lemma proves that all maps in $\Extcat_{\C}^1(b,a)$ are homotopy equivalences, whence is it a Kan complex \cite[Corollary~1.4]{Joy02}. 
	\end{remark} 
	
	The higher extension bifibrations are defined by compositing the bifibration $\epsilon_1$ with itself. Ayala--Francis devise composition procedure for bifibrations, rephrased here as \Cref{lem:AF20.5.2.1}. It uses the notion of localisation \cite[§7.1]{Cis19} \cite{DK80s,DK80c}. The \textit{Dwyer--Kan localisation} of a simplicial set ${X}$ with respect to a set of maps ${W}\subseteq {X}$ is given by the map $\ell_{W}$ in the following diagram
	\begin{equation*}
	\begin{tikzcd}
	{W} \arrow[d]\arrow[r,hook]\arrow[rd,phantom,"\ulcorner",very near end] & {X}\arrow[dd,bend left=49,"\ell_{{W}}"]\arrow[d] \\
	\Ex^{\infty}{W}  \arrow[r] & X'\arrow[d,"r"] \\
	 & {X}[{W}^{-1}]
	\end{tikzcd}
	\end{equation*}
	where the square is a pushout, and the map $r$ is a fibrant replacement in the Joyal model structure on simplicial sets. (Note that if $X$ is an $\infty$-category, one can choose $r$ to be the identity.) 
	Here, $\Ex^{\infty}$ is the transfinite composite of Kan's Ex-functor, which is the right adjoint of the barycentric subdivision functor  \cite{Kan57}. It is well-known that $\Ex^{\infty}$ is an endofunctor on the $\infty$-category of simplicial sets with essential image in the $\infty$-category of Kan complexes. Moreover, the unit $\upsilon\colon X\to \Ex^{\infty}X$ is a natural trivial cofibration in the category of simplicial sets. 
	 	
	\begin{lemma}[{\cite[Lemma 5.2.1]{AF20}}]\label{lem:AF20.5.2.1}
		Let $X_{01}\to \C\times \mathscr{I}$ and $X_{12}\to \mathscr{I} \times \mathscr{D}$ be bifibrations. Consider the pullback square
		\begin{equation}\label{eq:AF20.5.2.1}
		\begin{tikzcd}
		X_{012}\arrow[phantom,rd,very near start,"\lrcorner"] \arrow[r]\arrow[d] & X_{01}\arrow[d] \\
		X_{12} \arrow[r]  & \mathscr{I}
		\end{tikzcd}
		\end{equation}
		and the localisation \cite[§7.1]{Cis19} $X_{02}\coloneqq X_{012}[W^{-1}]$, where $W$ is the class of maps that are sent to invertible maps by both $X_{012}\to \C$ and $X_{012}\to \mathscr{D}$. Then the map $X_{02}\to \C\times\mathscr{D}$ is a bifibration.
	\end{lemma}

	In any context where \Cref{lem:AF20.5.2.1} applies, we call the maps in $W$ \textit{weak equivalences}. Denoting the composition operation of bifibrations by $\lozenge$, we define the extension bifunctor in degree $n$ by 
	\begin{equation*}
	\epsilon_n\coloneqq (\epsilon_1)^{\lozenge n} \colon \Extcat_{\C}^n(\C,\C) \to \C\times \C,
	\end{equation*}
	i.e. the $n$-th power of $\epsilon_1$. By \Cref{lem:AF20.5.2.1}, we can be certain that they are indeed bifibrations, whence we have bifunctors
	\begin{equation*}
	\Extcat_{\C}^n \coloneqq \epsilon_n^{-1} \colon \C\op\times\C\to \kan.
	\end{equation*}

	We can regard the objects in~$\Extcat^n_{\C}(\C,\C)$ as $n$ conjoined short exact sequences, as displayed below.
	\begin{equation}
\label{eq:compdia}
\begin{tikzcd}[scale cd= 0.8, column sep=0.75em]
& e_1 \arrow[rd, "p_1", two heads] &                                                       & e_2  \arrow[rd, "p_2", two heads] &                                                       & \cdots  \arrow[rd, "p_{n-2}", two heads] &                                                               & e_{n-1} \arrow[rd, "p_{n-1}", two heads] &                                                           & e_n \arrow[rd, "d_n", two heads] &   \\ a\arrow[rr, phantom, "\diamond"] \arrow[ru, "d_0", tail] \arrow[rd, two heads] &                                                    & e_{1.5}\arrow[rr, phantom, "\diamond"] \arrow[ru, "i_1", tail] \arrow[rd, two heads] &                                                    & e_{2.5}\arrow[rr, phantom, "\cdots"] \arrow[ru, "i_2", tail] \arrow[rd, two heads] &                                                               & e_{(n-2).5}\arrow[rr, phantom, "\diamond"] \arrow[ru, "i_{n-2}", tail] \arrow[rd, two heads] &                                                                & e_{(n-1).5}\arrow[rr, phantom, "\diamond"] \arrow[ru, "i_{n-1}", tail] \arrow[rd, two heads] &                                  & b \\
& 0 \arrow[ru, tail]                                 &                                                       & 0 \arrow[ru, tail]                                 &                                                       & \cdots \arrow[ru, tail]                                       &                                                               & 0 \arrow[ru, tail]                                             &                                                           & 0 \arrow[ru, tail]               &  
\end{tikzcd}
	\end{equation}
	As a matter of notation, we will omit the zero objects, thus deleting the bottom row. A map~$f\colon \mathbb{E}\to \mathbb{F}$ of such~$n$--extensions, or~$1$-simplex in~$\Extcat^n_{\C}(b,a)$, is then given by a homotopy coherent diagram
	\begin{equation}\label{eq:mapofnext}
		\begin{tikzcd}[scale cd= 0.8, row sep=1em]
			& e_1 \arrow[rd, "p^e_1", two heads] \arrow[dd, "\varepsilon_1"] &                                                   & e_2 \arrow[rd, "p^e_2", two heads] \arrow[dd, "\varepsilon_2"] &                                & e_n \arrow[rd, "d^e_n", two heads] \arrow[dd, "\varepsilon_n"] &              \\
			a_0\arrow[ru, "d^e_0"] \arrow[dd,"\alpha"] &                                                      & e_{1.5} \arrow[ru, "i^e_1"] \arrow[dd, "\varepsilon_{1.5}", near start] &                                                      & \cdots \arrow[ru, "i^e_{n-1}"] &                                                      & b_0 \arrow[dd,"\beta"] \\
			& f_1 \arrow[rd, "p^f_1", two heads]                   &                                                   & f_2 \arrow[rd, "p^f_2", two heads]                   &                                & f_n \arrow[rd, "d^f_n", two heads]                   &              \\
			a_1 \arrow[ru, "d^f_0"]            &                                                      & f_{1.5} \arrow[ru, "i^f_1"]                       &                                                      & \cdots \arrow[ru, "i^f_{n-1}"] &                                                      & b_1          
		\end{tikzcd}
	\end{equation}
	Such a map is a weak equivalence provided that $\alpha$ and $\beta$ are invertible.
	
	We should justify why our construction generalises MacLane's definition of extension categories for exact 1-categories. For any exact $\infty$-category $\C$, define $\classExtcat^n_\C(\C,\C)$ by $\Extcat^n_{\C}(\C,\C)$ for $n\leq 1$. The higher classical extension categories are recursively defined as pullbacks:
	\begin{equation*}
	\classExtcat^n_\C(\C,\C) \coloneqq \classExtcat^{n-1}_\C(\C,\C) \times_\C \classExtcat^1_\C(\C,\C)
	\end{equation*}
	Given $a,b\in\C$, one defines $\classExtcat_{\C}^n(b,a)$ in a similar manner as $\Extcat_{\C}^n(b,a)$. If $\C$ is the nerve of an exact 1-category, then $\classExtcat_{\C}^n(b,a)$ becomes equivalent to a classical extension category.
	
	We point out that $\Extcat^n_{\C}(b,a)$ will not be equivalent to MacLane's extension category $\classExtcat_{\C}^n(b,a)$ in general. In MacLane's setup, the higher extension categories (i.e. for $n\geq 2$) need not be Kan complexes. However, we will show in \Cref{proposition:is ClassicalExt} that $\Extcat^n_{\C}(b,a)$ is homotopy equivalent to $\classExtcat_{\C}^n(b,a)$. 

 	Let $a,b\in\C$. Then we can induce a map $\classExtcat^n_{\C}(b,a) \to^{\widehat{\ell}} \Extcat_{\C}^n(b,a)$ as a pullback of the natural map $\ell\colon \classExtcat^n_{\C}(\C,\C) \to \Extcat^n_{\C}(\C,\C)$. 
 	\begin{equation}\label{eq:ellandellhat}
 	\begin{tikzcd}
 	{\classExtcat^n_{\C}(b,a)} \arrow[d, dashed,"\widehat{\ell}"] \arrow[rr,hook]         &  & {\classExtcat^n_{\C}(\C,\C)} \arrow[d,"\ell"] \\
 	{\Extcat^n_{\C}(b,a)} \arrow[rrd,phantom,"\lrcorner",very near start] \arrow[rr,"\overline{(b,a)}",hook] \arrow[d] \arrow[rrd, phantom] &  & {\Extcat^n_{\C}(\C,\C)} \arrow[d]  \\
 	\Delta^0 \arrow[rr,"{(b,a)}",hook]                                            &  & \C\times\C                         
 	\end{tikzcd}
 	\end{equation}
 	Since the lower square and outer rectangle are cartesian, so is the upper square.

	\begin{proposition}
		\label{proposition:is ClassicalExt}
		If $\C$ is the nerve of an exact 1-category, then the map $\widehat{\ell}$ in \eqref{eq:ellandellhat} is a weak equivalence.
	\end{proposition}

A proof of \Cref{proposition:is ClassicalExt} will be given in the next section. One key ingredient will be the following result, where we show that extension categories of stable $\infty$-categories are mapping spaces.

\begin{proposition}\label{lem:weqHomnExtnSt}
	Suppose that $\D$ is an exact $\infty$-category admitting finite colimits,
	and let $\Sigma$ denote the suspension endofunctor on $\D$. Let $\map_{\D}(\D,\Sigma^n\D)$ denote the full subcategory of $\map(\Delta^1,\D)$ spanned by objects of the form $b\to \Sigma^n a$.
	For every $n\geq 0$, we have a weak homotopy equivalence of bifibrations
	\begin{equation}\label{eq:weqHomnExtnSt}
	\begin{tikzcd}
	{\Extcat^n_{\D}(\D,\D)} \arrow[rr, "\phi"] \arrow[rd, "\epsilon_n"'] &             & {\map_{\D}(\D,\Sigma^n\D)} \arrow[ld, "\mathrm{eval}_0\times \mathrm{eval}_1"] \\
	& \D\times \D &                                                                               
	\end{tikzcd}
	\end{equation}
	In particular, such weak homotopy equivalences exist if $\D$ is stable.
\end{proposition}
\begin{proof} 
	We have defined ${\map_{\D}(\D,\Sigma^n \D)}$ as a pullback, as displayed on the right in the following diagram
	\begin{equation}\label{eq:prweqHomnExtnSt}
	\begin{tikzcd}[column sep=3em]
	{\Extcat_{\D}^n(\D,\D)} \arrow[rd, "\epsilon_n"', bend right] \arrow[rr, "c", bend left] \arrow[r, "\phi", dashed] & {\map_{\D}(\D,\Sigma^n \D)} \arrow[d, "\mathrm{eval}_0\times \mathrm{eval}_1"'] \arrow[r, hook] \arrow[rd, phantom,"\lrcorner",very near start] & {\map(\Delta^1,\D)} \arrow[d, "\mathrm{eval}_0\times\mathrm{eval}_1"] \\
	& \D\times \D \arrow[r, "1\times \Sigma^n"]                                                                           & \D\times \D                                                          
	\end{tikzcd}
	\end{equation}
	where $c$ takes the cofibre of the last map of an $n$-extension.
	We induce $\phi$, which is clearly a map of bifibrations. The proof is completed using Quillen's Theorem A. Given an object $b\to^{f}\Sigma^n a$ in $\map_{\D}(\D,\Sigma^n \D)$, displayed to the right below
	\begin{equation*}
	\begin{tikzcd}
		& 0 \arrow[rd, two heads] \arrow[dd] &                                               & 0 \arrow[rd, two heads] \arrow[dd] &                                                 & \dots \arrow[rd, two heads]  &                                                    & e_n \arrow[rd, two heads] \arrow[dd] \arrow[rddd, "\square", phantom] &                   \\
		a \arrow[ru, tail] \arrow[dd, equal] &                                    & \Sigma a \arrow[ru, tail] \arrow[dd, equal] &                                    & \Sigma^2 a \arrow[ru, tail] \arrow[dd, equal] &                              & \Sigma^{n-1}a \arrow[ru, tail] \arrow[dd, equal] &                                                                       & b \arrow[dd, "f"] \\
		& 0 \arrow[rd, two heads]            &                                               & 0 \arrow[rd, two heads]            &                                                 & \cdots \arrow[rd, two heads] &                                                    & 0 \arrow[rd, two heads]                                               &                   \\
		a \arrow[ru, tail]                     &                                    & \Sigma a \arrow[ru, tail]                     &                                    & \Sigma^2 a \arrow[ru, tail]                     &                              & \Sigma^{n-1}a \arrow[ru, tail]                     &                                                                       & \Sigma^n a       
	\end{tikzcd}
	\end{equation*}
	one constructs the $n$-extension $\mathbb{E}$ displayed along the top row. Then $\mathbb{E}$ is sent to $f$ by $\phi$. Hence, we have identified a terminal object in the comma category $\phi\downarrow f$, completing the proof.
\end{proof}

Lastly in this section, we prove a useful and intuitively sensible result that will be helpful in later sections.

\begin{proposition}\label{prop:funcExtn}
	An exact functor 
	\begin{equation*}
	F\colon \mathscr{C} \to \mathscr{D}
	\end{equation*}
	induces maps of bifibrations
	\begin{equation}\label{eq:ExtF}
	\Extcat^nF\colon \Extcat^n_{\C}(\C,\C)\to \Extcat^n_{\mathscr{D}}(\mathscr{D},\mathscr{D})
	\end{equation}
	for every $n\geq 0$.
\end{proposition}
\begin{proof}
If $n=0$, we have defined $\Extcat_{\C}^0(\C,\C)=\map(\Delta^1,\C)$, whence we can induce
\begin{equation*}
\map(\C,\C) = \map(\Delta^1,\C) \to \map(\Delta^1,\mathscr{D})= \map(\mathscr{D},\mathscr{D}).
\end{equation*}
For $n=1$, we use a similar strategy; note that the exactness of $F$ implies that the induced functor
\begin{equation*}
\map(\Delta^1\times \Delta^1,\C) \to \map(\Delta^1\times \Delta^1,\mathscr{D}) 
\end{equation*}
restricts to the subcategories spanned by the short exact sequences. 

We complete the proof by induction. Consider the pullback squares
\begin{equation*}
\begin{tikzcd}
\mathscr{E}^n_{\mathscr{X}} \arrow[r]\arrow[d]\arrow[rd,phantom,"\lrcorner",very near start] & \Extcat^1_{\mathscr{X}}(\mathscr{X},\mathscr{X})\arrow[d] \\
\Extcat^{n-1}_{\mathscr{X}}(\mathscr{X},\mathscr{X})\arrow[r] & \mathscr{X}
\end{tikzcd}
\end{equation*}
where $\mathscr{X}$ can be either $\C$ or $\mathscr{D}$. A map $\mathscr{E}^n_{\C}\to \mathscr{E}^n_{\mathscr{D}}$ is induced by the universal property. Since this preserves weak equivalences (indeed, the functor $F$ preserves homotopy equivalences), we induce a map
\begin{equation*}
\Extcat_{\C}^n(\C,\C) \to \Extcat_{\mathscr{D}}^n(\mathscr{D},\mathscr{D}),
\end{equation*}
as desired.
\end{proof}

\section{Retakh spectra for exact $\infty$-categories}\label{sec:retakh}

In this section, we generalise the following theorem of Retakh to the setting of exact $\infty$-categories. 

\begin{theorem}[\cite{Ret86}]\label{thm:Ret86}
	Let $\C$ be (the nerve of) an exact 1-category, and let $a$ and $b$ be objects in $\C$. Then the spaces $\{\classExtcat^n_{\C}(b,a)\}_{n\geq 0}$ form an $\Omega$-spectrum, which is to say that we have weak equivalences
	\begin{equation}\label{eq:Ret86}
	R_n\colon \classExtcat^{n+1}_{\C}(b,a) \to \Omega \classExtcat^{n}_{\C}(b,a)
	\end{equation}
	for all $n\geq 0$.
\end{theorem}

Given an exact $\infty$-category $\C$, the $\Omega$-spectra $\{\Extcat^n_{\C}(b,a)\}_{n\geq 0}$ will be referred to as \textit{Retakh spectra}. 
We first address the special case of stable $\infty$-categories, where the extension $\infty$-categories are determined by hom-spaces (see \Cref{lem:weqHomnExtnSt}). 

\begin{lemma}\label{lem:LurieHomSpec}
	Let $\D$ be an exact $\infty$-category admitting finite colimits (for example a stable $\infty$-category), and let $a$ and $b$ be objects in $\D$. Then the Kan complexes $\{\Extcat^n_{\D}(b,a)\}$ form an $\Omega$-spectrum.
\end{lemma}
\begin{proof}
	By \Cref{lem:weqHomnExtnSt}, we need to establish the existence of weak equivalences
	\begin{equation}\label{eq:LurieHomSpec}
	\map_{\D}(b,\Sigma^{n+1} a) \to \Omega \map_{\D}(b,\Sigma^n a)
	\end{equation}
	These weak equivalences exist, since $\D$ admits finite colimits \cite[p. 24]{Lur17}.
\end{proof}

The proof of \Cref{lem:LurieHomSpec} suggests that \Cref{thm:Ret86} can be strengthened. Indeed, the equivalences \eqref{eq:LurieHomSpec} are natural in both components. We will set out to prove that any exact $\infty$-category $\C$ determines a functor
\begin{equation*}
\C\op\times\C\to \Sp,
\end{equation*}
where the $n$-th component is weakly equivalent to
\begin{equation*}
\Extcat^n_{\C}\colon \C\op\times \C \to \kan.
\end{equation*}
Here, the co-domain $\Sp$ is the $\infty$-category of spectrum objects in Kan complexes \cite[Section~8]{Lur06}. The existence of such a bifunctor implies that \Cref{thm:Ret86} can be generalised to exact $\infty$-categories, and moreover that the maps $R_n$ in \eqref{eq:Ret86} will be natural in both components.

The remainder our this section is devoted to proving our main theorem, which is stated presently.

\begin{theorem}\label{thm:SpExtbifun}
	Let $\C$ be an exact $\infty$-category. There exists a functor
	\begin{equation}\label{eq:SpExtbifun}
	\Extcat_{\C}\colon\C\op\times\C\to \Sp
	\end{equation}
	such that the $n$-th component of the spectrum $\Extcat_{\C}(b,a)$ is weakly equivalent to the extension $\infty$-category $\Extcat^n_{\C}(b,a)$ for all $a,b\in \C$.
\end{theorem}

The following corollaries will be immediately follow. 

\begin{corollary}\label{thm:IsOmegaSpectrum}
	Let $\C$ be an exact $\infty$-category.
	For every pair~$(b,a)$ of objects in $\C$, the sequence
	\begin{equation*}
	\Extcat_\C(b,a) \coloneqq \{\Extcat^n_{\C}(b,a)\}_{n\geq 0}
	\end{equation*}
	is an $\Omega$-spectrum.
\end{corollary}

\begin{corollary}
	The~$(-n)$-th homotopy group of~$\Extcat_\C(b,a)$ is naturally isomorphic to $\pi_0 \Extcat_\C^n (b,a)$.
\end{corollary}
To show that \Cref{thm:SpExtbifun} holds, we will prove a stronger statement which better reflects our proof strategy.

\begin{proposition}\label{lem:weqHomnExtn}
	Let $\C$ be an exact $\infty$-category. There exists an exact $\infty$-category $\D$, admitting finite colimits, and an exact embedding $\C\embed \D$ inducing weak homotopy equivalences (see \Cref{prop:funcExtn})
	\begin{equation}\label{eq:weqHomnExtn}
	\Extcat^n_{\C}(\C,\C) \to \Extcat^n_{\D}(\C,\C) =\map_{\D}(\C,\Sigma^n \C), 
	\end{equation}
	where $\Extcat^n_{\D}(\C,\C)$ is defined by the pullback
	\begin{equation}\label{eq:ExtDCC}
	\begin{tikzcd}
		\Extcat^n_{\D}(\C,\C) \arrow[rd,phantom,"\lrcorner",very near start]\arrow[d]\arrow[r,hook] & \Extcat^n_{\D}(\D,\D)\arrow[d] \\
		\C\times\C \arrow[r,hook] & \D\times\D
	\end{tikzcd}
	\end{equation}
	and $\map_{\D}(\C,\Sigma^n \C)$ is the full subcategory of $\map_{\D}(\D,\Sigma^n \D)$ (see \eqref{eq:prweqHomnExtnSt}) spanned by the objects $b\to \Sigma^n a$, where $a,b\in\C$. 
\end{proposition}

Once \Cref{lem:weqHomnExtn} is shown to hold, one readily proves \Cref{thm:SpExtbifun} using \Cref{lem:weqHomnExtnSt}. Indeed, if $\D$ admits finite colimits, we can take the bifunctor $\C\op\times\C\to \Sp$ as the mapping space bifunctor. To find a suitable $\D$, we use the explicit construction of Klemenc' hull \cite{Kle20}.

When $\C$ is the nerve of an exact 1-category $\c$, \Cref{lem:weqHomnExtn} is proven using a classical result of Verdier, stating that the Yoneda extension group $\classExtcat^n_{\c}(b,a)$ is naturally isomorphic to $\hom_{\mathrm{D}^b(\c)}(b,\Sigma^n a)$ \cite[Prop~3.2.2]{Ver96} \cite[A.7, Proposition]{Pos11}, where $\mathrm{D}^b(\c)$ is the bounded derived category of $\c$.

Not only do we rely on the existence of a stable hull to prove \Cref{thm:SpExtbifun}, but also on its structure. More precisely, the exact $\infty$-category $\D$ in \Cref{prop:equivExt} will be a certain subcategory of the stable hull of $\C$. We collect some useful facts on its construction in \Cref{lem:Klemenc}, so that our the proof of \Cref{lem:CintoHstgeq0} makes sense. 

Klemenc' notation should be recalled first. Given an exact $\infty$-category $\C$, let $\mathrm{PSh}(\C)$ be the $\infty$-category of presheaves on $\C$. Let $P_{\Sigma,f}(\C)$ be the smallest subcategory of $\mathrm{PSh}(\C)$ which is closed under finite colimits and contains the image of the Yoneda embedding
\begin{equation*}
\widehat{(-)}\colon \C \embed \mathrm{PSh}(\C),
\end{equation*}
and let $P_{\Sigma}(\C)$ be the full subcategory of $\mathrm{PSh}(\C)$ spanned by the presheaves preserving finite products. Both $P_{\Sigma,f}(\C)$ and $P_{\Sigma}(\C)$ are $\infty$-categories admitting finite colimits. The \textit{Spanier--Whitehead $\infty$-category} of $P_{\Sigma,f}(\C)$, denoted by $\mathcal{SW}P_{\Sigma,f}(\C)$, is defined as the colimit of the following digram of $\infty$-categories
\begin{equation*}
\begin{tikzcd}
P_{\Sigma,f}(\C) \arrow[r,"\Sigma"] & P_{\Sigma,f}(\C) \arrow[r,"\Sigma"] & \cdots
\end{tikzcd}
\end{equation*}
where the arrows denote the suspension endofunctor on $P_{\Sigma,f}$. 
We have an embedding $P_{\Sigma,f}(\C)\embed \mathcal{SW}P_{\Sigma,f}(\C)$.
Klemenc defines primitive acyclic objects as certain colimits in $P_{\Sigma,f}(\C)$ constructed from short exact sequences in $\C$ \cite[Definition 3.7]{Kle20}. In $\mathcal{SW}P_{\Sigma,f}(\C)$, an \textit{acyclic object} is one that is contained in the stable closure of the primitive acyclic objects. A \textit{quasi-isomorphism} in $\mathcal{SW}P_{\Sigma,f}(\C)$ admits an acyclic cofibre.

\begin{lemma}\label{lem:Klemenc}
	Let $\C$ be an exact $\infty$-category and let $P_{\Sigma,f}(\C)$, $P_{\Sigma}(\C)$, and $\mathcal{SW}P_{\Sigma,f}(\C)$ be defined as in the paragraph above.
	\begin{enumerate}
		\item \cite[Theorem 1.2]{Kle20} The stable hull $\Hst(\C)$ of $\C$ can be defined as the localisation of $\mathcal{SW}P_{\Sigma,f}(\C)$ with respect to the quasi-isomorphisms.
		\item \cite[Proof of Proposition 4.7]{Kle20} The homotopy category $\homcat \Hst(\C)$ admits a left calculus of fractions.
	\end{enumerate}
\end{lemma}

Write $\mathcal{H}^{\mathrm{st}}_{\geq 0}(\C)$ for the localisation of $P_{\Sigma,f}(\C)$ in the quasi-isomorphisms.
This is an exact $\infty$-category in which maps onto zero objects are cofibrations. It will be shown in \Cref{prop:equivExt} that we can set $\D=\mathcal{H}^{\mathrm{st}}_{\geq 0}(\C)$ 
in \Cref{lem:weqHomnExtn}. We will need to adapt a lemma of Klemenc'.

\begin{lemma}\label{lem:CintoHstgeq0}
	Let $\C$ be an exact $\infty$-category and let $\Hst(\C)$ be its stable hull. For any fibration $E\epi^{p}\widehat{b}$ in $\mathcal{H}^{\mathrm{st}}_{\geq 0}(\C)$, where $b$ is an object in $\C$, there exist a map $\widehat{y_\ast}\to^{t} E$ in $\mathcal{H}^{\mathrm{st}}_{\geq 0}(\C)$,
	where $y_\ast$ is an object in $\C$ and the composite $\widehat{y_\ast}\epi^{pt}\widehat{b}$ is a fibration in $\C$.
\end{lemma}
\begin{remark}\label{rem:CintoHstgeq}
	The proof of \Cref{lem:CintoHstgeq0} closely resembles Klemenc' proof that an exact $\infty$-category is extension closed in its stable hull \cite[Proposition~4.25]{Kle20}. There will be a subtle departure in the final steps. For the convenience of the reader, our proof is carefully spelled out in full. 
\end{remark}
\begin{proof}[Proof of \Cref{lem:CintoHstgeq0}]
	Let $K\to^{j} E $ be the fibre of $p$ in $\mathcal{H}_{\geq 0}^{\mathrm{st}}(\C)$. 
	In the homotopy category $\homcat\mathcal{H}^{\mathrm{st}}_{\geq 0}(\C)$, we can represent the fibre sequence 
	\begin{equation*}
	\begin{tikzcd}
	K\arrow[r,"j"] & E\arrow[r,"p"] & \widehat{b}
	\end{tikzcd}
	\end{equation*}
	in terms of the left calculus of fractions as follows:
	\begin{equation*}
	\begin{tikzcd}
	K \arrow[rd, "j'"', dashed] \arrow[r, "j"] & E' \arrow[r]                                     & \cofibre(j) \arrow[d, "\sim"',"f"] \\
	& E \arrow[rd, dashed,"p"'] \arrow[u, "\sim"] \arrow[r] & B                              \\
	&                                                  & \widehat{b} \arrow[u, "\sim","g"'] 
	\end{tikzcd}
	\end{equation*}
	The dashed arrows represent maps on the $\infty$-categorical level, and the solid arrows are in $P_{\Sigma,f}$. Arrows labeled by $\sim$ are quasi-isomorphisms. Let $ B\to^{d} Q $ be the cofibre of $f$ in $P_{\Sigma,f}$. Since $f$ is a quasi-isomorphism, the object $Q$ is acyclic. Thus, the homotopy group $\pi_0(Q)$ is effaceable \cite[Proposition~3.15]{Kle20}, which is to say that there is a fibration $w\epi^{q} v$ in $\C$ such that $\pi_0(Q)$ is the cokernel of $q$ in $\Mod(\homcat\C)$ under the embedding
	\begin{equation*}
	\overline{(-)}\colon \homcat\C\embed \Mod(\homcat \C).
	\end{equation*}
	Since we have an epimorphism $\overline{v}\epi \pi_0(Q)$ in $ \Mod(\homcat \C)$, there exists a map $b\to^{h} v$ such that the following diagram commutes.
	\begin{equation*}
	\begin{tikzcd}
	&                               & \overline{v} \arrow[d, two heads] \\
	\overline{b} \arrow[r, "\pi_0g"'] \arrow[rru, "\overline{h}", dashed] & \pi_0(B) \arrow[r, "\pi_0d"'] & \pi_0(Q)                         
	\end{tikzcd}
	\end{equation*}
	Let $y_{\ast}\epi^{p_{\ast}} b$ be the pullback in $\C$ of $q$ along $h$. As the composite $d\circ g\circ p_{\ast}$ is null-homotopic, we can perform the following lifting in $P_{\Sigma,f}$:
	\begin{equation*}
	\begin{tikzcd}
	&                               & \cofibre(j) \arrow[d, "f","\sim"'] \\
	\widehat{y_{\ast}} \arrow[r, "\widehat{p}_\ast"',two heads] \arrow[rru, dashed] & \widehat{b} \arrow[r, "g"'] & B                        
	\end{tikzcd}
	\end{equation*}
	Since $K\in \mathcal{H}_{\geq 0}^{\mathrm{st}}(\C)$, there are homotopy classes of maps from $\widehat{b}$ to $\Sigma K$ \cite[Proposition 3.4]{Kle20} (this is where our proof  departs from Klemenc'). Hence a composite of $\widehat{y_\ast} \epi \widehat{b} \to^{g} B$ factors through $E$. 
	We conclude that there exists a commutative square
	\begin{equation}\label{eq:CintoHstgeq0d}
	\begin{tikzcd}
	\widehat{y_\ast} \arrow[d] \arrow[r, two heads,"\widehat{y_{\ast}}"] & \widehat{b} \arrow[d, "\sim"',"g"] \\
	E \arrow[r, two heads]                     & B                            
	\end{tikzcd}
	\end{equation}
	in $\mathcal{H}^{\mathrm{st}}_{\geq 0}(\C)$, whence the proof is complete.
\end{proof}

We will now make the statement in \Cref{lem:weqHomnExtn} more precise, and then prove it.

\begin{lemma}\label{prop:equivExt}
	Let $\C$ be an exact $\infty$-category, and let $\mathcal{H}^{\mathrm{st}}_{\geq 0}(\C)$ be defined as above. Then the inclusion
	\begin{equation}\label{eq:prop:equivExt}
	\Extcat^n_{\C}(\C,\C) \embed \Extcat^n_{\mathcal{H}^{\mathrm{st}}_{\geq 0}(\C)}(\C,\C)
	\end{equation}
	is a weak homotopy equivalence, where the co-domain is defined by the pullback in \eqref{eq:ExtDCC} with $\D=\mathcal{H}^{\mathrm{st}}_{\geq 0}(\C)$.
\end{lemma}
\begin{proof}
	Define $\classExtcat^n_{\mathcal{H}^{\mathrm{st}}_{\geq 0}(\C)}(\C,\C)$ by the pullback shown as the upper right square in the following diagram
	\begin{equation}\label{eq:classExtcatH}
	\begin{tikzcd}
	{\classExtcat^n_{\C}(\C,\C)} \arrow[d] \arrow[r,hook,"j"] \arrow[rd, phantom,"\lrcorner",very near start] & {\classExtcat^n_{\mathcal{H}^{\mathrm{st}}_{\geq 0}(\C)}(\C,\C)} \arrow[d] \arrow[r, hook] \arrow[rd, "\lrcorner", phantom, very near start]                                                   & {\classExtcat^n_{\mathcal{H}^{\mathrm{st}}_{\geq 0}(\C)}(\mathcal{H}^{\mathrm{st}}_{\geq 0}(\C),\mathcal{H}^{\mathrm{st}}_{\geq 0}(\C))} \arrow[d] \\
	\C^{\times (n+1)} \arrow[r,hook]                                         & \C\times\mathcal{H}^{\mathrm{st}}_{\geq 0}(\C)\times \cdots \times \mathcal{H}^{\mathrm{st}}_{\geq 0}(\C)\times \C \arrow[r, hook] \arrow[d] \arrow[rd, "\lrcorner", phantom, very near start] & \mathcal{H}^{\mathrm{st}}_{\geq 0}(\C)^{\times (n+1)} \arrow[d]                                                                                    \\
	& \C\times\C \arrow[r, hook]                                                                                                                                                                     & \mathcal{H}^{\mathrm{st}}_{\geq 0}(\C)\times \mathcal{H}^{\mathrm{st}}_{\geq 0}(\C)                                                               
	\end{tikzcd}
	\end{equation} 
	In the bottom square, the vertical maps project onto the outermost factors. Let $W_{\C}$ denote the subcategory of ${\classExtcat^n_{\C}(\C,\C)}$ spanned by the weak equivalences, and let $W_{\mathcal{H}^{\mathrm{st}}_{\geq 0}(\C)}$ denote the similarly defined subcategory of ${\classExtcat^n_{\mathcal{H}^{\mathrm{st}}_{\geq 0}(\C)}(\C,\C)}$. Then $\Extcat^n_{\mathcal{H}^{\mathrm{st}}_{\geq 0}(\C)}(\C,\C)$ is the localisation of $\classExtcat^n_{\mathcal{H}^{\mathrm{st}}_{\geq 0}(\C)}(\C,\C)$ with respect to $W_{\mathcal{H}^{\mathrm{st}}_{\geq 0}(\C)}$. We show that the map $j$ above is a weak homotopy equivalence which restricts to a weak equivalence  $W_{\C}\to W_{\mathcal{H}^{\mathrm{st}}_{\geq 0}(\C)}$. It will follow that the induced map between the localisations, displayed in \eqref{eq:prop:equivExt}, is a weak homotopy equivalence \cite[6.2]{DK80s}.
	
	Let $E\epi^{p} b$ be a fibration in $\mathcal{H}^{\mathrm{st}}_{\geq 0}(\C)$ with $b\in\C$, and let $y_{\ast}\epi^{p_{\ast}} b$ be the fibration constructed in \Cref{lem:CintoHstgeq0}. Then $p_{\ast}$ is an object in an $\infty$-category $\mathrm{fib}_{\C}(p)$, which is the full subcategory of the over-category of $p$ in $\map(\Delta^1,\mathcal{H}^{\mathrm{st}}_{\geq 0}(\C))$ spanned by fibrations in $\C$ with co-domain $b$ (see \eqref{eq:CintoHstgeq0d}). It will be useful to prove the fact that $p_{\ast}$ is a terminal object in $\mathrm{fib}_{\C}(p)$. This is to say that a simplicial sphere in $\mathrm{fib}_{\C}(p)$ can be filled as long as its last object is $p_{\ast}$ \cite[Definition~4.1]{Joy02}. We first prove the claim for $1$-spheres. A map of fibrations
	\begin{equation*}
	\begin{tikzcd}
	{y'} \arrow[d] \arrow[r, two heads,"p'"] & {b} \arrow[d,equal] \\
	E \arrow[r, two heads,"p"]                     & {b}                           
	\end{tikzcd}
	\end{equation*}
	can be embedded into the solid part of the diagram
	\begin{equation*}
	\begin{tikzcd}
	w \arrow[r, two heads]                                                                       & v \arrow[r]                              & \Sigma K' \arrow[r]                                                              & \Sigma w                            \\
	y_{\ast} \arrow[r, "p_\ast", two heads] \arrow[d] \arrow[u] \arrow[ru, "\square", phantom] & b \arrow[d, equal] \arrow[u] \arrow[r] & \Sigma K' \arrow[r] \arrow[u, equal] \arrow[d] \arrow[rd, "\star", phantom] & \Sigma y_{\ast} \arrow[d] \arrow[u] \\
	E \arrow[r, "p", two heads]                                                                  & b \arrow[d, equal] \arrow[r]           & \Sigma K \arrow[r]                                                               & \Sigma E                            \\
	y' \arrow[r, "p'", two heads] \arrow[uuu,dashed,bend left=61]\arrow[uu,dotted,bend left=30] \arrow[u]                                                      & b                                        &                                                                                  &                                    
	\end{tikzcd}
	\end{equation*}
	where all rows but the bottom are cofibre sequences. By \Cref{Buh2.12}, the square marked by $\star$ is bicartesian.
	Any choice of composite $y'\to b \to  v\to \Sigma K'$ is null-homotopic, as it solves the following pullback problem:
	\begin{equation*}
	\begin{tikzcd}
	y_{\ast} \arrow[rdd, "0"', bend right] \arrow[rrd, "0", bend left] \arrow[rd, dashed] &                                                               &                           \\
	& \Sigma K' \arrow[r] \arrow[d] \arrow[rd, "\star", phantom] & \Sigma y_{\ast} \arrow[d] \\
	& \Sigma K \arrow[r]                                            & \Sigma E                 
	\end{tikzcd}
	\end{equation*}
	We induce the dashed arrow $y' \to w$, and hence the dotted arrow $y'\to y_\ast$ by the universal property of the pullback. A simplicial 1-sphere has been filled. We only relied on universal properties, whence on can use the same type of argument to fill higher dimensional spheres.

	Any object in $\classExtcat^n_{\mathcal{H}^{\mathrm{st}}_{\geq 0}(\C)}(\C,\C)$ is weakly equivalent to an object in $\classExtcat^n_{\C}(\C,\C)$. Indeed, one uses \Cref{lem:CintoHstgeq0} to construct it, as shown in the diagram below. The diagram is supposed to read from right to left, and then from bottom to top. Objects denoted by minuscule letters are in $\C$ and those represented by capital letters need not be.
	In the last step, we use that $\C$ is closed under extensions in $\mathcal{H}^{\mathrm{st}}_{\geq 0}(\C)$; thus $e_1\in \C$.
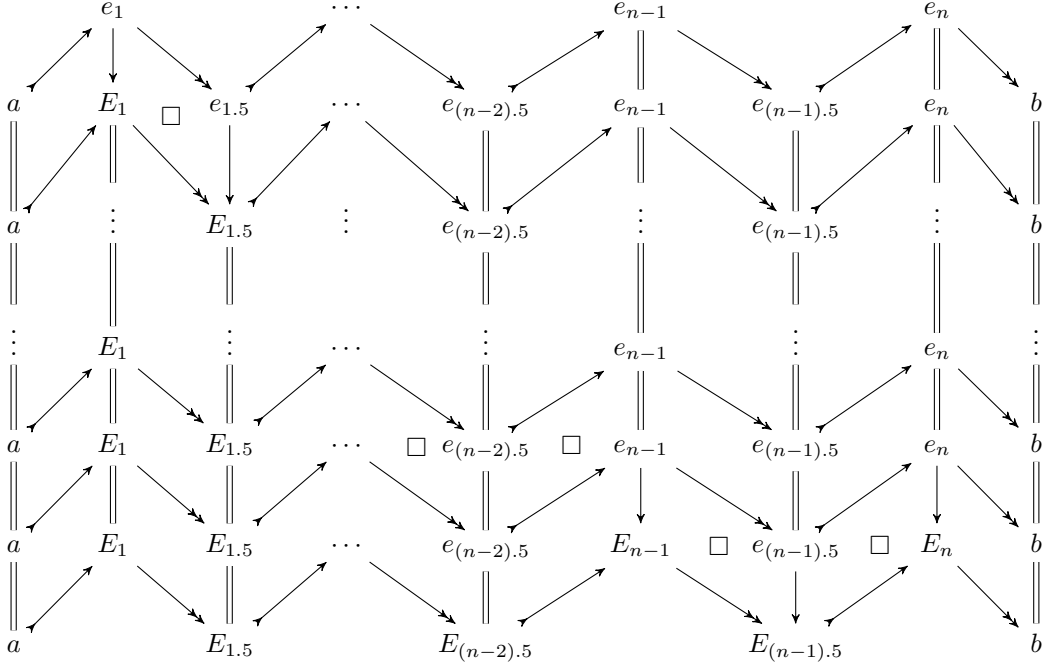
	\begin{equation*}
	\begin{tikzcd}[row sep=2em, column sep=2em]
	& e_1 \arrow[d] \arrow[rd, two heads]          &                                                                             & \cdots \arrow[rd, two heads]                                 &                                                 & e_{n-1} \arrow[d, equal] \arrow[rd, two heads]                                                      &                                                                      & e_n \arrow[d, equal] \arrow[rd, two heads] &                           \\
	a \arrow[d,equal] \arrow[ru, tail]          & E_1 \arrow[d, equal] \arrow[rd, two heads] & e_{1.5} \arrow[d] \arrow[ru, tail]                                          & \cdots \arrow[rd, two heads]                                 & e_{(n-2).5} \arrow[d, equal] \arrow[ru, tail] & e_{n-1} \arrow[d, equal] \arrow[rd, two heads]                                                      & e_{(n-1).5} \arrow[d, equal] \arrow[ru, tail]                      & e_n \arrow[d, equal] \arrow[rd, two heads] & b \arrow[d, equal]      \\
	a \arrow[d, equal] \arrow[ru, tail] & \vdots \arrow[d, equal]                    & E_{1.5} \arrow[d, equal] \arrow[ru, tail] \arrow[luu, "\square", phantom] & \vdots                                                       & e_{(n-2).5} \arrow[d, equal] \arrow[ru, tail] & \vdots \arrow[d, equal]                                                                             & e_{(n-1).5} \arrow[d, equal] \arrow[ru, tail]                      & \vdots \arrow[d, equal]                    & b \arrow[d, equal]      \\
	\vdots \arrow[d, equal]             & E_1 \arrow[rd, two heads] \arrow[d, equal] & \vdots \arrow[d, equal]                                                   & \cdots \arrow[rd, two heads] \arrow[rdd, "\square", phantom] & \vdots \arrow[d, equal]                       & e_{n-1} \arrow[rd, two heads] \arrow[d, equal]                                                      & \vdots \arrow[d, equal]                                            & e_n \arrow[rd, two heads] \arrow[d, equal] & \vdots \arrow[d, equal] \\
	a \arrow[ru, tail] \arrow[d, equal] & E_1 \arrow[rd, two heads] \arrow[d, equal] & E_{1.5} \arrow[ru, tail] \arrow[d, equal]                                 & \cdots \arrow[rd, two heads]                                 & e_{(n-2).5} \arrow[ru, tail] \arrow[d,equal]          & e_{n-1} \arrow[rd, two heads] \arrow[d] \arrow[rdd, "\square", phantom] \arrow[l, "\square", phantom] & e_{(n-1).5} \arrow[ru, tail] \arrow[d, equal]                      & e_n \arrow[rd, two heads] \arrow[d]          & b \arrow[d, equal]      \\
	a \arrow[ru, tail] \arrow[d, equal] & E_1 \arrow[rd, two heads]                    & E_{1.5} \arrow[ru, tail] \arrow[d, equal]                                 & \cdots \arrow[rd, two heads]                                 & e_{(n-2).5} \arrow[ru, tail] \arrow[d, equal] & E_{n-1} \arrow[rd, two heads]                                                                         & e_{(n-1).5} \arrow[ru, tail] \arrow[d] \arrow[r, "\square", phantom] & E_n \arrow[rd, two heads]                    & b \arrow[d, equal]      \\
	a \arrow[ru, tail]                    &                                              & E_{1.5} \arrow[ru, tail]                                                    &                                                              & E_{(n-2).5} \arrow[ru, tail]                    &                                                                                                       & E_{(n-1).5} \arrow[ru, tail]                                         &                                              & b                        
	\end{tikzcd}
	\end{equation*}	
	Let $\mathbb{E}$ denote the $n$-extension on the bottom and $\mathbb{E}'$ the $n$-extension top row.
	From the assertion shown in the previous paragraph, one deduces $\mathbb{E}'$ is universal among the $n$-extensions in $\classExtcat_{\C}(\C,\C)$ mapping to $\mathbb{E}$.
	One uses Quillen's Theorem A to conclude that both 
	\[j\colon \classExtcat_{\C}(\C,\C)\to \classExtcat_{\mathcal{H}^{\mathrm{st}}_{\geq 0}(\C)}(\C,\C)\ \] and the restriction $j\colon W_{\C}\to W_{\mathcal{H}^{\mathrm{st}}_{\geq 0}(\C)}$ are weak homotopy equivalences, whence the proof is complete.
\end{proof}

We are now ready to prove the main theorem.

\begin{proof}[Proof of \Cref{thm:SpExtbifun}]
	Consider the exact embedding $\C\embed \mathcal{H}^{\mathrm{st}}_{\geq 0}(\C)$. We define $\Extcat_{\C}$ as the composite
	\begin{equation*}
	\begin{tikzcd}[column sep=5em]
	\C\op\times\C  \arrow[r,hook,"\eta_{\C}\times\eta_{\C}"] & \mathcal{H}_{\geq 0}^{\mathrm{st}}(\C)\op\times \mathcal{H}_{\geq 0}^{\mathrm{st}}(\C)\arrow[r,"\map_{\mathcal{H}_{\geq 0}^{\mathrm{st}}(\C)}"] & \Sp
	\end{tikzcd}
	\end{equation*}
	where $\map_{\mathcal{H}^{\mathrm{st}}_{\geq 0}(\C)}$ is the mapping spectrum functor of $\mathcal{H}^{\mathrm{st}}_{\geq 0}(\C)$. The assertion now follows from \Cref{lem:CintoHstgeq0} and \Cref{prop:equivExt}, since $\mathcal{H}_{\geq 0}^{\mathrm{st}}(\C)$ admits finite colimits.
\end{proof}

\begin{proof}[Proof of \Cref{proposition:is ClassicalExt}]
		Let $\delta\colon \Delta^1 \to \Delta^n$ be the embedding that sends the edge $01\in \Delta^1$ to $0n\in \Delta^n$. This is a trivial cofibration in the classical model structure of simplicial sets. One induces a map
	\begin{equation*}
	\map(\Delta^n,\mathcal{H}^{\mathrm{st}}_{\geq 0}(\C)) \to^{\delta^{\ast}} \map(\Delta^1,\mathcal{H}^{\mathrm{st}}_{\geq 0}(\C))
	\end{equation*}
	which, since $\mathcal{H}^{\mathrm{st}}_{\geq 0}(\C)$ is an $\infty$-category, is a trivial Kan fibration \cite[Second assertion in Corollary 3.6.4 with $Y=\Delta^0$]{Cis19}. From \Cref{lem:weqHomnExtnSt} and \Cref{prop:equivExt}, one deduces that the $\infty$-category $\Extcat_{\C}^n(\C,\C)$ is equivalent to the full subcategory of $\map(\Delta^1,\mathcal{H}^{\mathrm{st}}_{\geq 0}(\C))$ spanned by objects of the form $b\to \Sigma^n a$, where $a,b\in \C$. In other words, we have a pullback
	\begin{equation*}
	\begin{tikzcd}[column sep=5em]
	{\Extcat^n_{\C}(\C,\C)} \arrow[d] \arrow[r,hook]\arrow[rd,phantom,"\lrcorner",very near start] & {\map(\Delta^1,\mathcal{H}^{\mathrm{st}}_{\geq 0}(\C))} \arrow[d,"\mathrm{eval}_0\times \mathrm{eval}_1"]                   \\
	\C\times\C \arrow[r,"\eta_{\C} \times \Sigma^n\eta_{\C}",hook]                & \mathcal{H}^{\mathrm{st}}_{\geq 0}(\C)\times \mathcal{H}^{\mathrm{st}}_{\geq 0}(\C)
	\end{tikzcd}
	\end{equation*}
	In a similar vein, one can also construct a diagram of cartesian squares 
	\begin{equation*}
	\begin{tikzcd}[column sep=4em]
	{\classExtcat^n_{\C}(\C,\C)} \arrow[rd,phantom,"\lrcorner",very near start] \arrow[r,hook]\arrow[d]& {\classExtcat^n_{\mathcal{H}^{\mathrm{st}}_{\geq 0}(\C)}(\C,\C)} \arrow[d] \arrow[r,hook]\arrow[rd,phantom,"\lrcorner",very near start] & {\map(\Delta^n,\mathcal{H}^{\mathrm{st}}_{\geq 0}(\C))} \arrow[d,"\mathrm{eval}_0\times \cdots \times \mathrm{eval}_{n-1}"']                   \\
	\C \times \cdots \times \C \arrow[r,hook]  & \C \times \mathcal{H}^{\mathrm{st}}_{\geq 0}(\C) \times  \cdots \times \mathcal{H}^{\mathrm{st}}_{\geq 0}(\C) \times \C \arrow[r,hook]              &  \mathcal{H}^{\mathrm{st}}_{\geq 0}(\C)^{\times n}
	\end{tikzcd}
	\end{equation*}
	where $\classExtcat_{\mathcal{H}^{\mathrm{st}}_{\geq 0}(\C)}(\C,\C)$ is defined by the upper right pullback square in \Cref{eq:classExtcatH}. We can now assemble a homotopy coherent cube
	\begin{equation*}
	\begin{tikzcd}[column sep = 0.5em]
	{\classExtcat^n_{\mathcal{H}^{\mathrm{st}}_{\geq 0}(\C)}(\C,\C)} \arrow[rd] \arrow[rrr] \arrow[dddd] \arrow[rrrdddd,, very near start, phantom]                                        &                                                               &  & {\map(\Delta^n,\mathcal{H}^{\mathrm{st}}_{\geq 0}(\C))} \arrow[rd,"\delta^{\ast}"] \arrow[dddd,"\mathrm{eval}_0\times \cdots \times \mathrm{eval}_{n-1}",near end] &                                                                                     \\
	& {\Extcat^n_{\C}(\C,\C)} \arrow[rrr,crossing over]   \arrow[rrrdddd, very near start, phantom]             &  &                                                                                 & {\map(\Delta^1,\mathcal{H}^{\mathrm{st}}_{\geq 0}(\C))} \arrow[dddd,,"\mathrm{eval}_0\times \mathrm{eval}_1",near end]                \\
	&                                                               &  &                                                                                 &                                                                                     \\
	&                                                               &  &                                                                                 &                                                                                     \\
	\C \times \mathcal{H}^{\mathrm{st}}_{\geq 0}(\C) \times  \cdots \times \mathcal{H}^{\mathrm{st}}_{\geq 0}(\C) \times \C \arrow[rrr,near end] \arrow[rd] \arrow[rrrrd, phantom, very near start] &                                                               &  & \mathcal{H}^{\mathrm{st}}_{\geq 0}(\C)^{\times n} \arrow[rd,"\pi"]                    &                                                                                     \\
	& \C\times \C \arrow[rrr, "\eta_{\C} \times \Sigma^n\eta_{\C}"] \arrow[uuuu,crossing over,<-] &  &                                                                                 & \mathcal{H}^{\mathrm{st}}_{\geq 0}(\C)\times \mathcal{H}^{\mathrm{st}}_{\geq 0}(\C)
	\end{tikzcd}
	\end{equation*}
	where $\pi$ projects onto the outermost factors.
	As the front, back, and lower faces are cartesian, it follows from the pasting law that the top face is cartesian. We draw a diagram of cartesian squares
	\begin{equation*}
	\begin{tikzcd}
	{\classExtcat_{\C}^n(b,a)} \arrow[d,hook,"\widehat{j}"] \arrow[r]\arrow[dd,bend right=80,"\widehat{\ell}"'] \arrow[rd, "\lrcorner", phantom,very near start] & {\classExtcat_{\C}^n(\C,\C)} \arrow[d,hook,"j"] & \\
	{\classExtcat_{\mathcal{H}^{\mathrm{st}}_{\geq 0}(\C)}^n(b,a)} \arrow[d,"m"] \arrow[r,hook] \arrow[rd, "\lrcorner", phantom,very near start] & {\classExtcat_{\mathcal{H}^{\mathrm{st}}_{\geq 0}(\C)}^n(\C,\C)} \arrow[d] \arrow[r]\arrow[rd, "\lrcorner", phantom,very near start] & {\map(\Delta^n,\mathcal{H}^{\mathrm{st}}_{\geq 0}(\C))} \arrow[d,"\delta^{\ast}"] \\
	{\Extcat_{\C}^n(b,a)} \arrow[r]                                                 & {\Extcat_{\C}^n(\C,\C)} \arrow[r]                & {\map(\Delta^1,\mathcal{H}^{\mathrm{st}}_{\geq 0}(\C))}          
	\end{tikzcd}
	\end{equation*}
	where the pasted rectangle on the left is the same as the top square in \eqref{eq:ellandellhat}. 
	Since $\delta^\ast$ is a trivial Kan fibration, it is a proper map \cite[Proposition 4.4.11]{Cis19} and it has contractible fibres \cite[\href{https://kerodon.net/tag/0074}{Tag 0074}]{Ker}.
	It follows that the map \[\classExtcat_{\mathcal{H}^{\mathrm{st}}_{\geq 0}(\C)}^n(b,a) \to^{m} \Extcat_{\C}^n(b,a)\] is a localisation with respect to the maps in $\classExtcat_{\C}^n(b,a)$ that are sent to invertible maps in $\Extcat_{\C}^n(b,a)$ \cite[Proposition 7.1.12]{Cis19}. We showed in the proof of \Cref{lem:CintoHstgeq0} that the map
	\[j\colon \classExtcat_{\C}(\C,\C)\to \classExtcat_{\mathcal{H}^{\mathrm{st}}_{\geq 0}(\C)}(\C,\C)\ \]
	is a weak homotopy equivalence. Using the same methods, one can show that the same is true for $\widehat{j}$. We conclude that $\widehat{\ell}$ is a localisation.
\end{proof}

We conclude this section by spectrifying \Cref{prop:funcExtn}. This will be the key ingredient when we describe topological enhancements of extriangulated functors.

\begin{proposition}\label{prop:nattransExt}
	An exact functor 
	\begin{equation*}
	F\colon \mathscr{C} \to \mathscr{D}
	\end{equation*}
	induces a canonical natural transformation 
	\begin{equation}\label{eq:natF}
	\eta_F\colon \Extcat_{\C}(-,-)\Longrightarrow \Extcat_{\mathscr{D}}(F-,F-)
	\end{equation}
	of functors~$\C\op\times\C\to\Sp$.
\end{proposition}
\begin{proof}
	The natural transformation \eqref{eq:natF} should be defined in terms of a map of the bifibrations that make up its components.
	Consider the homotopy pullback diagram
	\begin{equation*}
	\begin{tikzcd}[ampersand replacement=\&]
	P^n\arrow[r,"\widehat{\epsilon_n^{\D}}"]\arrow[d]\arrow[dr, phantom, "\lrcorner",very near start]\& \C\times\C \arrow[d,"F\times F"] \\
	\Extcat^n_{\mathscr{D}}(\mathscr{D},\mathscr{D})\arrow[r,"e_{n}^{\D}"] \& \mathscr{D}\times\mathscr{D}
	\end{tikzcd}
	\end{equation*}
	where the bottom row is the bifibration $\mathrm{eval}_{0}\times \mathrm{eval}_{n}$.
	The functor $\widehat{\epsilon_n^{\D}}$ is then a bifibration corresponding to the bifunctor $\Extcat_{\mathscr{D}}(F-,F-)$. An appropriate morphism of fibrations is now induced by the universal property of pullbacks, as displayed in the diagram below.
	\begin{equation*}
	\begin{tikzcd}[ampersand replacement=\&]
	\Extcat^n_{\C}(\C,\C) \arrow[rd,"\epsilon_n^{\C}",bend left]\arrow[dd,bend right=49,"\Extcat^n\! F"']\arrow[d,dashed] \\		
	P^n\arrow[r,"\widehat{e^{\D}_n}"]\arrow[d]\arrow[dr, phantom, "\lrcorner",very near start]\& \C\times\C \arrow[d,"F\times F"] \\
	\Extcat^n_{\mathscr{D}}(\mathscr{D},\mathscr{D})\arrow[r,"\epsilon_n^{\mathscr{D}}"] \& \mathscr{D}\times\mathscr{D}
	\end{tikzcd}
	\end{equation*}
	Here, the map $\Extcat^n F$ is the induced map established by \Cref{prop:funcExtn}.
	The proof is complete.
\end{proof}

	\section{Extriangulated homotopy categories and higher extensions}\label{section:extriangulated}
	
	In \Cref{thm:SpExtbifun}, we showed that every exact $\infty$-category $\C$ determines a bifunctor
	\begin{equation*}
	\Extcat_{\C}\colon \C\op\times\C \to \Sp,
	\end{equation*}
	sending a pair of objects to the Retakh spectrum. Taking the $(-1)$-st homotopy group lets us induce a bifunctor on the homotopy category:
		\begin{equation*}
	\pi_{-1}\Extcat_{\C}\colon (\homcat\C)\op\times\homcat\C \to \Ab.
	\end{equation*}
	It has recently been shown that $\homcat \C$ has a natural extriangulated structure when equipped with the bifunctor~$\pi_{-1}\Extcat_{\C}$ as well as a readily accessible additive realisation \cite[Theorem~4.22]{NP20}. In this section, we make use of the Retakh Ext-$\Omega$-spectrum to shed new light on Nakaoka--Palu's result. More specifically, our claim is that the Retakh spectrum determines the extriangulation of the homotopy category as well as extriangulated functors between them. We will also prove in \Cref{thm:higherext} that higher extension categories induce the higher extension groups when passing to the homotopy category.
	
	We need a few preliminary definitions to review the relatively novel theory of extriangulated categories.
	
	\begin{definition}
		Let~$\c$ be an additive category and let~$\mathbf{E}\colon \c\op \times \c \to \Ab$ be an additive bifunctor into the category of abelian groups. We refer to group elements~$\xi\in \mathbf{E}(b,a)$ as \emph{$\mathbf{E}$-extensions}. A \emph{morphism of~$\mathbf{E}$-extensions}~$\xi\in\mathbf{E}(b_1,a_1)$ and~$\xi'\in\mathbf{E}(b_2,a_2)$ is a pair of morphisms~$(\alpha\colon a_1\to a_2,\beta\colon b_1\to b_2)$ such that~$\mathbf{E}(\beta,a_2)(\xi')=\mathbf{E}(b_1,\alpha)(\xi)$.
	\end{definition}
	
	Having defined a notion of a morphism of~$\mathbf{E}$-extensions, it is quickly checked that there is a category of such. We denote this by~$\mathbf{E}\mhyphen\!\Ext_{\c}^1$. 
	
	The role played by~$\mathbf{E}$--extensions is that of equivalence classes of exact sequences. If~$\C$ is an exact $\infty$-category, consider the functor
	\begin{equation*}
	\Extcat_{\C}^1\colon \C\op\times\C \to \kan
	\end{equation*}
	from which we can induce an additive bifunctor
	\begin{equation*}
	\mathbf{E}\coloneqq\pi_{0}\Extcat_{\C}^1\colon (\homcat\C)\op\times\homcat\C\to \Ab,
	\end{equation*}
	sending a tuple~$(b,a)$ to the set of equivalence classes of exact sequences in~$\C$ of the form~$ a\mono e\epi b$ (i.e. to a connected component of~$\Extcat^1_{\C}(b,a)$). An~$\mathbf{E}$-extension is simply an equivalence class of exact sequences. Given morphisms $\alpha\colon a_1\to a_2$ and $\beta\colon b_1\to b_2$, the homomorphism~$\mathbf{E}(b,\alpha)$ sends an exact sequence~$a\mono e\epi b$ to the bottom row of the following commutative diagram
	\begin{center}
		\begin{tikzcd}
			a\arrow[r]\arrow[d,"\alpha"]\arrow[dr, phantom, "\square"] &e\arrow[r]\arrow[d] & b\arrow[d,equal] \\
			a'\arrow[r] & z\arrow[r] & b
		\end{tikzcd}
	\end{center}
	and dually for~$\mathbf{E}(\beta,a)$. \Cref{Buh3.1} then gives that a morphism of~$\mathbf{E}$-extensions is an equivalence class of morphisms of complexes
	\begin{center}
		~$(\alpha,\varphi,\beta)\colon$
		\begin{tikzcd}
			a\arrow[r]\arrow[d,"\alpha"] & e\arrow[r]\arrow[d,"\varphi"] & b\arrow[d,"\beta"] \\
			a'\arrow[r] & e'\arrow[r] & b' \\
		\end{tikzcd}
	\end{center}
	which is uniquely determined by the pair~$(\alpha,\beta)$.
	
	\begin{lemma}\label{exfunclem}
		Let~$\mathbf{E}\colon\c\op\times\c\to\Ab$ and~$\mathbf{E}'\colon\c{'}\op\times\c'\to\Ab$ be additive bifunctors. If~$F\colon\c\to\c'$ is an additive functor, then a natural transformation~$$\eta\colon\mathbf{E}(-,-)\to\mathbf{E}'(F\op-,F-)$$ induces a functor 
		$F_{\eta}\colon \mathbf{E}\mhyphen\!\Ext_{\c}\to \mathbf{E}'\mhyphen\!\Ext_{\c'}$. If~$F$ is an equivalence and~$\eta$ is a natural isomorphism, then~$F_{\eta}$ is an equivalence. 
	\end{lemma}
	\begin{proof}
		The maps~$\eta_{b,a}\colon\mathbf{E}(b,a)\to\mathbf{E}'(Fb,Fa)$ establish a suitable map of objects. If~$(\alpha,\beta)\colon \xi_1\to\xi_2$ is a morphism of~$\mathbf{E}$-extensions~$\xi_1\in\mathbf{E}(b_1,a_1)$ and~$\xi_2\in\mathbf{E}(b_2,a_2)$, the commutativity of the diagram
		\begin{center}
			\begin{tikzcd}
				\mathbf{E}(b_1,a_1)\arrow[r,"\eta_{b_1,a_1}"]\arrow[d,"{\mathbf{E}(b_1,\alpha)}"'] & \mathbf{E}'(Fb_1,Fa_1)\arrow[d,"{\mathbf{E}'(Fb_1,\alpha)}"]\\
				\mathbf{E}(b_1,a_2)\arrow[r,"\eta_{b_1,a_2}"]&\mathbf{E}'(Fb_1,Fa_2) \\
				\mathbf{E}(b_2,a_2)\arrow[r,"\eta_{b_2,a_2}"]\arrow[u,"{\mathbf{E}(\beta,a_2)}"]&\mathbf{E}'(Fb_2,Fa_2)\arrow[u,"{\mathbf{E}'(\beta,Fa_2)}"']
			\end{tikzcd}
		\end{center}
		shows that~$(F\alpha,F\beta)$ is a morphism~$\eta_{b_1,a_1}\xi_1\to\eta_{b_2,a_2}\xi_2$ of~$\mathbf{E}'$-extensions. Since identity morphisms are preserved, and one readily checks that composition is respected by adding a column to the diagram above, we draw the conclusion that $F_{\eta}$ is indeed a functor.
		
		If~$F$ is an equivalence and~$\eta$ is a natural isomorphism, it is clear that~$F_{\eta}$ is essentially surjective. Full fidelity follows from the full fidelity of~$F$.
	\end{proof}
	
	Since we have a natural isomorphism~$\pi_{-1}\Extcat_{\C}\to \pi_{0}\Extcat^1_{\C}$, between the bifunctors~$(\homcat\C)\op\times\homcat \C\to\Ab$, \Cref{exfunclem} provides an equivalence 
	\begin{equation}\label{eq:natiso}
	\Phi\colon\mathbf{E}'\mhyphen\!\Ext_{\homcat\C} \to \mathbf{E}\mhyphen\!\Ext_{\homcat\C}
	\end{equation}
	where~$\mathbf{E}=\pi_{0}\Extcat^1_{\C}$ and~$\mathbf{E}'=\pi_{-1}\Extcat_{\C}$.
	
	Unlike~$\pi_0\Extcat_{\C}^1$, an arbitrary additive bifunctor~$\mathbf{E}$ need not be directly tied to exact sequences. In the general case, it necessary to manually link the elements of the abstract group~$\mathbf{E}(b,a)$ to concrete diagrams of the form~$a\to e\to b$, or rather equivalence classes of such. 
	
	\begin{definition}
		Let~$\c$ be an additive category. Two diagrams of the form~$a\to e\to b$ and \\ $a\to e'\to b$ are said to be \emph{equivalent} if there exists an isomorphism \\ $\varphi\colon e\to e'$ rendering the following diagram commutative.
		\begin{center}
			\begin{tikzcd}
				&e\arrow[dr]\arrow[dd,"\varphi",]& \\
				a\arrow[ur]\arrow[dr] && b\\
				&e'\arrow[ur]&
			\end{tikzcd}
		\end{center}
		The equivalence class containing~$a\to e\to b$ is denoted~$[a\to e\to b]$.
	\end{definition}
	
	It will be convenient to regard the equivalence classes of sequences as objects in a category~$\seq_{\c}$. 
	A morphism in this category from~$[a\to e\to b]$ to~$[a'\to e'\to b']$ is pair of morphisms~$(\alpha\colon a\to a',\beta\colon b\to b')$
	for which there exists a commutative diagram as follows:
	\begin{center}
		\begin{tikzcd}
			a\arrow[r]\arrow[d,"\alpha"] & e\arrow[r]\arrow[d,"\varphi"] & b\arrow[d,"\beta"] \\
			a'\arrow[r] & e'\arrow[r] & b' \\
		\end{tikzcd}
	\end{center}
	Composition in this category is clearly well defined and associative. 
	
	The attentive reader might have had a \textit{dejà vu} experience when reading the last paragraph. Indeed, the category of~$\pi_0\Extcat_{\C}^1$-extensions is remarkably similar to~$\seq_{\homcat\C}$. The difference is that~$\seq_{\homcat\C}$ contains more objects, giving rise to an embedding~$\iota_{\C}\colon\mathbf{E}\mhyphen\!\Ext^1_{\homcat\C}\embed \seq_{\homcat\C}$, where~$\mathbf{E}=\pi_0\Extcat_{\C}^1$. This is an example of the general notion of realisation.
	
	\begin{definition}
		Let~$\mathbf{E}\colon \c\op \times \c \to \Ab$ be an additive bifunctor.
		A \emph{realisation} is a functor 
		\begin{equation*}
		\mathfrak{s}\colon\mathbf{E}\mhyphen\!\Ext_{\c}^1\to \seq_{\c}
		\end{equation*}
		such that an~$\mathbf{E}$-extension~$\xi\in \mathbf{E}(b,a)$ is sent to an equivalence class~$[a\to e\to b]$ with appropriate endpoints. 
		Equivalently, a \emph{realisation} is a injection~$\mathfrak{s}$ that sends an~$\mathbf{E}$--extension~$\xi\in \mathbf{E}(b,a)$ to an equivalence class~$[a\to e\to b]$, satisfying the following property: If~$(\alpha,\beta)$ is a morphism of~$\mathbf{E}$-extensions that are realised by~$[a\to e\to b]$ and~$[a'\to e'\to b']$, there exists a morphism~$\varphi\colon e\to 
		e'$ such that the diagram
		\begin{center}
			\begin{tikzcd}
				a\arrow[r]\arrow[d,"\alpha"] & e\arrow[r]\arrow[d,"\varphi"] & b\arrow[d,"\beta"] \\
				a'\arrow[r] & e'\arrow[r] & b' \\
			\end{tikzcd}
		\end{center}
		is commutative. 
		
		A realisation~$\mathfrak{s}$ is \emph{additive} if~$\mathfrak{s}(_{b}0_{a})=[\sigma_1(b,a)]$ (where ~$_{b}0_{a}$ is the zero element in~$\mathbf{E}(b,a)$ and~$\sigma_1(b,a)$ is the split exact sequence) and~$\mathfrak{s}(\xi_1\oplus\xi_2)=\mathfrak{s}(\xi_1)\oplus\mathfrak{s}(\xi_2)$. 
	\end{definition}
	
	A realisation restricts to a map from~$\mathbf{E}(b,a)$ to a subset of the set of equivalence classes of the form~$[a\to e\to b]$. Since~$\mathbf{E}(b,a)$ is an abelian group, this map transports a group structure to this subset. Moreover, an additive realisation gives rise to a group structure that behaves similarly to Yoneda Ext-groups in exact categories.  
	
	The embedding~$\iota_{\C}\colon\mathbf{E}\mhyphen\!\Ext^1_{\homcat\C}\embed \seq_{\homcat\C}$, where~$\mathbf{E}=\pi_{0}\Extcat^1_{\C}$, was our motivating example of a realisation. It is easy to see that this is additive. Composing with the above equivalence \eqref{eq:natiso} yields an additive realisation 
	\begin{equation*}
	\mathbf{E}'\mhyphen\!\Ext^1_{\homcat\C}\to\mathbf{E}\mhyphen\!\Ext^1_{\homcat\C}\embed \seq_{\homcat\C}
	\end{equation*}
	of~$\mathbf{E}'=\pi_{-1}\Extcat_{\C}$.
	
	We now have all ingredients to define what is meant by an extriangulated category and extriangulated functors. 
	
	\begin{definition}[{\cite[Definition~2.12]{NP19}}]\label{extricat}
		Let~$\c$ be an additive category, and consider a bifunctor~$\mathbf{E}\colon\c\op\times\c\to\Ab$ as well as an additive realisation~$\mathfrak{s}$ of~$\mathbf{E}$. The triple~$(\c,\mathbf{E},\mathfrak{s})$ is an \emph{extriangulated category} if the following axioms hold.
		\begin{enumerate}[label=(ET\arabic*), leftmargin=1.42cm]
			\setcounter{enumi}{2}
			\item[(ET3)$\;\;\;$]\label{ET3} Let~$\xi_1\in \mathbf{E}(b_1,a_1)$ and~$\xi_2\in \mathbf{E}(b_2,a_2)$ be~$\mathbf{E}$-extensions with realisations 
			\begin{equation*}
			\mathfrak{s}(\xi_1) = [a_1\to e_1\to b_1] \qquad \mathfrak{s}(\xi_2) = [a_2\to e_2\to b_2]
			\end{equation*}
			If the solid part of the following diagram commutes
			\begin{center}
				\begin{tikzcd}
				a_1\arrow[r]\arrow[d] & e_1\arrow[r]\arrow[d] & b_1\arrow[d,dashed] \\
					a_2\arrow[r] & e_2\arrow[r] & b_2
				\end{tikzcd}
			\end{center}
			the dashed morphism exists and the resulting diagram is commutative.
			\item[(ET3$\op$)]\label{ET3op} Dual of the above \cite[(ET3$\op$) in Definition~2.12]{NP19}.
			\item[(ET4)$\;\;\;$]\label{ET4} Let~$\xi_1\in \mathbf{E}(b,a)$ and~$\xi_2\in \mathbf{E}(c,e)$ be~$\mathbf{E}$-extensions with realisations 
			\begin{equation*}
			\mathfrak{s}(\xi_1) = [a\to e\to b], \qquad \mathfrak{s}(\xi_2) = [e\to f\to c].
			\end{equation*}
			There exists a commutative diagram 
			\begin{center}
				\begin{tikzcd}
					a\arrow[r,"\iota"]\arrow[d,equal] & e\arrow[r,"\pi"]\arrow[d] & b\arrow[d,"\iota'"] \\
					a\arrow[r] & f\arrow[r]\arrow[d] & g\arrow[d,"\pi'"] \\
					& c\arrow[r,equal] & c
				\end{tikzcd}
			\end{center}
			and an~$\mathbf{E}$-extension~$\xi_3\in\mathbf{E}(g,a)$ which is realised by the middle row. Moreover, we require that
			\begin{enumerate}[label=(ET4.\arabic*), leftmargin=1.38cm]
				\item~$\mathfrak{s}(\mathbf{E}(c,\pi)(\xi_2))=[b\to^{\iota'} g\to^{\pi'} c]$,
				\item~$\mathbf{E}(\iota',a)(\xi_3)=\xi_1$,
				\item~$\mathbf{E}(g,\iota)(\xi_3)=\mathbf{E}(\pi',e)(\xi_2)$.
			\end{enumerate}
			\item[(ET4$\op$)]\label{ET4op} Dual of the above \cite[Remark~2.22]{NP19}.
		\end{enumerate}
	\end{definition}
	
	\begin{definition}[{\cite[Definition~2.32]{B-TS20}}]\label{extrifunc}
		Let~$(\c,\mathbf{E},\mathfrak{s})$ and~$(\c',\mathbf{E}',\mathfrak{s}')$ be extriangulated categories. An \emph{extriangulated functor} consists of an additive functor~$F\colon \c\to \c'$ for which there exists a natural transformation~$\eta \colon \mathbf{E}(-,-)\to \mathbf{E}'(F\op-,F-)$ such that the diagram of functors
		\begin{center}
			\begin{tikzcd}
				\mathbf{E}\mhyphen\!\Ext_{\c} \arrow[r,"F_{\eta}"]\arrow[d,"\mathfrak{s}"] & \mathbf{E}'\mhyphen\!\Ext_{\c'}\arrow[d,"\mathfrak{s}'"]\\
				\seq_{\c} \arrow[r,"\seq_{F}"] & \seq_{\c'}
			\end{tikzcd}
		\end{center}
		commutes, where~$F_{\eta}$ is as in \Cref{exfunclem} and~$\seq_{F}$ sends the class~$[a\to^{\iota} e\to^{\pi} b]$ to the class \\ $[Fa\to^{F\iota} Fe\to^{F\pi} Fb]$.
		An extriangulated functor is an \emph{equivalence} if~$F$ is an equivalence of categories. 
	\end{definition}
	
	The homotopy category of an exact $\infty$-category admits a natural extriangulation. More precisely:
	
	\begin{theorem}[{\cite[Theorem~4.22 and Proposition~4.28]{NP20}}]\label{NP20.4.22}
		Let~$\C$ be an exact $\infty$-category, and consider the additive bifunctor
		\begin{equation*}
		\mathbf{E}=\pi_0\Extcat_{\C}^1\colon (\homcat\C)\op\times \homcat\C \to \Ab,
		\end{equation*}
		as well as the additive realisation 
		\begin{equation*}
		\mathbf{E}=\pi_{0}\Extcat^1_{\C}\embed^{\iota_{\C}} \seq_{\homcat\C}.
		\end{equation*}
		Then the triple~$(\homcat\C,\mathbf{E},\iota_{\C})$ is an extriagulated category. Furthermore, if~$F\colon \C\to \D$ is an exact functor, then the induced functor~$\homcat F\colon \homcat\C\to \homcat\D$ between homotopy categories is an extriangulated functor.
	\end{theorem}
	
	An extriangulated category is \textit{topological} if it is equivalent to the homotopy category of some exact $\infty$-category.
	
	We noted above that we have a natural isomorphism of bifunctors
	\begin{equation*}
 	\pi_{-1}\Extcat_{\C} \Longrightarrow \pi_0\Extcat_{\C}^1.
	\end{equation*}
	A realisation of~$\pi_{-1}\Extcat_{\C}(-,-)$ is achieved by precomposing the realisation~$\iota_{\C}$ with the equivalence $\Phi$ from \eqref{eq:natiso}. It now follows from \Cref{exfunclem} that the identity functor~$1_{\homcat\C}$ on~$ \homcat\C$ yields an extriangulated equivalence 
	\begin{equation*}
	(\homcat\C,\pi_{-1}\Extcat_{\C},\iota_{\C}\circ\Phi)\to (\homcat\C,\pi_0\Extcat_{\C}^1,\iota_{\C}).
	\end{equation*}
	In other words, we can present the extriangulation on the topological extriangulated category~$\homcat\C$ using the Retakh spectrum. 
	
	Let~$F\colon \C\to \D$ be an exact functor of exact $\infty$-categories. By \Cref{prop:nattransExt} we have a natural transformation 
	\begin{equation*}
	\Extcat_{\C}(-,-)\Longrightarrow \Extcat_{\mathscr{D}}(F-,F-)
	\end{equation*}
	of functors~$\C\op\times\C\to\Sp$. We can thus induce a natural transformation 
	\begin{equation*}
	\eta\colon\pi_{-1}\Extcat_{\C}(-,-)\Longrightarrow \pi_{0}\Extcat^1_{\mathscr{D}}(F-,F-)
	\end{equation*}
	of functors~$(\homcat\C)\op\times\homcat\C\to\Ab$. Pushing this natural transformation down to the homotopy category ensures that the functor
	\begin{equation*}
	\homcat F\colon \homcat \C\to \homcat\D
	\end{equation*}
	is extriangulated. 
	Indeed, by \Cref{exfunclem}, we have a functor 
	\begin{equation*}
	\homcat F_{\eta}\colon \pi_{-1}\Extcat_{\C}\mhyphen\! \Ext \to \pi_{-1}\Extcat_{\C}\mhyphen\! \Ext,
	\end{equation*}
	since the vertical maps in the diagram
	\begin{center}
		\begin{tikzcd}
			\pi_{-1}\Extcat_{\C}\mhyphen\! \Ext \arrow[r,"\homcat F_{\eta}"]\arrow[d] & \pi_{-1}\Extcat_{\D}\mhyphen\! \Ext\arrow[d]\\
			\seq_{\homcat\C} \arrow[r,"\seq_{F}"] & \seq_{\homcat\D}
		\end{tikzcd}
	\end{center}
	sends an equivalence class of extensions~$[a\mono e\epi b]$ to itself, it is easy to see that the diagram commutes, and consequently that~$\homcat F$ is an extriangulated functor. 
	
	Klemenc' embedding theorem (restated here as \Cref{thm:Kle20.1}), as well as the discussion above, proves that any small topological extriangulated category embeds into a triangulated category.
	
	As a new contribution, we generalise \Cref{NP20.4.22}. The spectrum objects~$\Extcat_{\C}(b,a)$ do not only determine a natural extriangulation on~$\homcat \C$. They also induce the higher extension functors, recently defined by Gorsky--Nakaoka--Palu \cite[Definition~3.1]{GNP21}. We review their definition first.
	
	Let~$F$ and~$G$ be bifunctors~$\c\op\times\c \to \Ab$. For a fixed pair of objects~$(x,y)$, consider the bifunctor~$S\coloneqq G(-,y)\underset{\Z}{\otimes} F(x,-)$. Using the coend of~$G$ and~$F$, one defines $G\diamond F$ as follows:
	\begin{equation*}
	(G\diamond F)(x,y) \coloneqq \int\limits^{c\in\c} G(c,y)\underset{\Z}{\otimes} F(x,c).
	\end{equation*}
	Since the category of abelian groups is cocomplete, the coend of a bifunctor~$S$ can be obtained as the coequaliser of the diagram
	\begin{equation}\label{eq:coenddef}
	\begin{tikzcd}[column sep = 6em]
	{\coprod\limits_{f\colon c_2\rightarrow c_1} S(c_2,c_1)} \arrow[r,shift left,"\coprod{S(c_2,f)}"]\arrow[r,shift right,"\coprod{S(f,c_1)}"'] & \coprod\limits_{c} S(c,c),
	\end{tikzcd}
	\end{equation}
	where the first coproduct is indexed over all morphisms in~$\c$, and the second over all its objects.

	Let~$(\C,\mathbf{E},\mathfrak{s})$ be an extriangulated category. We set the zeroth extension functor to be the Hom-functor~$\hom_{\c}(-,-)$, and the first to be~$\mathbf{E}$.
	The higher extension functors~$\mathbf{E}^{\diamond n}$ are then defined inductively by the formula \cite[Definition~3.1]{GNP21}
	\begin{equation*}
	\mathbf{E}^{\diamond n} \coloneqq \mathbf{E}^{\diamond n-1}\diamond \mathbf{E}.
	\end{equation*}
	
	This section is concluded with its main result.
	
	\begin{theorem}\label{thm:higherext}
		Let~$\C$ be an exact $\infty$-category. 
		The bifunctor 
		\begin{equation*}
		\pi_0\Extcat^n_{\C}\colon (\homcat\C)\op\times \homcat\C \to \Ab
		\end{equation*}
		is naturally isomorphic to the~$n$th extension functor~$\pi_{0}\Extcat^1_{\C}(-,-)^{\diamond n}$.
	\end{theorem}
	\begin{proof}
		We proceed by induction. The base cases will be~$n=0$ and~$n=1$, both of which hold by definition.
		
		Suppose that we have natural isomorphisms
		\[
		\Phi^n_{x,y}\colon \pi_{0}\Extcat^1_{\C}(x,y)^{\diamond n} \to \pi_0\Extcat^n_{\C}(x,y),
		\]
		for some~$n\geq 1$. Applying the tensor functor gives another natural isomorphism
		\begin{equation*}
		\begin{tikzcd}
		\Phi^n_{x,y}\otimes 1\colon \pi_{0}\Extcat^1_{\C}(x,y)^{\diamond n}\underset{\Z}{\otimes} \pi_{0}\Extcat^1_{\C}(b,c) \arrow[r] & \pi_0\Extcat^n_{\C}(x,y) \underset{\Z}{\otimes} \pi_{0}\Extcat^1_{\C}(b,c).
		\end{tikzcd}
		\end{equation*}
		Consider also the homomorphism
		\begin{equation*}
		\begin{tikzcd}
		\gamma_{b,a}\colon\pi_0\Extcat^n_{\C}(c,a) \underset{\Z}{\otimes} \pi_{0}\Extcat^1_{\C}(b,c)\arrow[r] & \pi_0\Extcat^{n+1}_{\C}(b,a)
		\end{tikzcd}
		\end{equation*}
		that concatenates exact sequences.
		Our strategy to complete the inductive step will be to show that the composite $\gamma_{b,a}\circ (\Phi^n_{c,a}\otimes 1)$ is a coequaliser of the diagram \eqref{eq:coenddef}, with the tensor product~$\pi_0\Extcat^n_{\C}(-,a)\underset{\Z}{\otimes} \pi_{0}\Extcat^1_{\C}(b,-)$ in place of~$S$.
		We simplify our notation by setting~$S=\pi_0\Extcat^n_{\C}(-,a)\underset{\Z}{\otimes} \pi_{0}\Extcat^1_{\C}(b,-)$ for the remainder of this proof.
		
		For this composite to be a coequaliser, it must be the case that 
		\begin{equation}\label{eq:iscofork}
		\gamma_{b,a}\circ (\Phi^n_{c,a}\otimes 1)\circ S(c_2,f) = \gamma_{b,a}\circ (\Phi^n_{c,a}\otimes 1)\circ S(f,c_1)
		\end{equation}
		for all maps $f\colon c_2\to c_1$. In other words, it should be shown that 
		\begin{equation*}
		\gamma_{b,a}\circ (\Phi^n_{c,a}\otimes 1)
		\end{equation*}
		is a cofork of the diagram \eqref{eq:coenddef}. 
		The situation is illustrated by the solid part of the following diagram
		\begin{equation}\label{eq:iscoforkdia}
		\begin{tikzcd}
		a \arrow[r, dashed] \arrow[d, dashed,"1"] & \cdots \arrow[r, dashed] & e_{(n-1).5} \arrow[r, dashed] \arrow[d,dashed,"1"] & c' \arrow[r, dashed] \arrow[d, dashed] \arrow[dr, phantom, "\lrcorner",very near start] & c_2 \arrow[d, "f"'] \arrow[r] \arrow[dr, phantom, "\ulcorner",very near end] & g \arrow[r] \arrow[d, dashed] & b \arrow[d, dashed,"1"] \\
		a \arrow[r]                           & \cdots \arrow[r]         & e_{(n-1).5} \arrow[r]                   & e_n \arrow[r]                          & c_1 \arrow[r, dashed]         & g' \arrow[r, dashed]          & b                  
		\end{tikzcd}
		\end{equation}
		The diagram is completed by taking a pullback and pushout of~$f$ along the morphisms~$e_n\to c_1$ and~$c_2\to g$, respectively, and then adding identity morphisms so that we end up with a morphism of exact sequences. 
		The top row of the entire diagram then displays a representative of an equivalence class in the left hand side of \eqref{eq:iscofork}, and the bottom row displays one the right hand side. Since the diagram \eqref{eq:iscoforkdia} is a morphism between them in~$\Extcat^{n+1}_{\C}(b,a)$, these representatives belong to the same equivalence class, as desired.

		Lastly, we explain why~$\gamma_{b,a}\circ (\Phi^n_{c,a}\otimes 1)$ is a universal cofork of \eqref{eq:coenddef}. Given an abelian group~$M$ and a group homomorphism 
		\begin{equation*}
		\begin{tikzcd}
		\coprod\limits_{c} S(c,c)\arrow[r,"m"] & M
		\end{tikzcd}
		\end{equation*}
		satisfying~$m\circ S(c_2,f) = m\circ S(f,c_1)$ for all $f\colon c_2\to c_1$, our task is to find a homomorphism~$$\overline{m}\colon \pi_0\Extcat^{n+1}_\C(b,a)\to M$$ such that~$m=\overline{m}\circ c_{b,a}\circ (\varphi^n_{c,a}\otimes 1)$, and then show that~$\overline{m}$ is unique. Let~$\overline{m}(\mathbb{E})\coloneqq m(\widetilde{\mathbb{E}})$, where~$\widetilde{\mathbb{E}}$ is a choice of preimage of~$\mathbb{E}$ in~$\coprod\limits_c S(c,c)$. This is a well-defined homomorphism; a different choice of preimage would yields the same result. The uniqueness of~$\overline{m}$ is a direct consequence of the fact that~$c_{b,a}\circ (\varphi^n_{c,a}\otimes 1)$ is an epimorphism.
	\end{proof}
	
	\appendix
	\section{Proofs of diagram lemmas}\label{app:diaproofs}
	
	We ended \Cref{sec:exqcat} with a series of diagram lemmas for exact categories, claiming that they generalise to exact $\infty$-categories. These results are found in Bühler's monograph \cite{Buh10}, as well as countless other texts. The proofs presented here are in an $\infty$-categorical context. Unlike the special case of exact 1-categories, cofibrations (resp. fibrations) are not necessarily monomorphisms (resp. epimorphisms) in a general $\infty$-categorical context.
	
	\begin{replemma}{5lem}
		Consider a map of exact sequences
		\begin{center}
			\begin{tikzcd}
				a \arrow[r, "i_a",tail] \arrow[d,"f'"]   & e \arrow[d,"f"]\arrow[r,two heads,"p_b"] & b\arrow[d,"f''"] \\
				
				c \arrow[r,tail,"i_c"]               & f \arrow[r,two heads,"p_d"] & d                                  
			\end{tikzcd}
		\end{center}
		If~$f'$ and~$f''$ are homotopy equivalences (resp. cofibrations, resp fibrations), so is~$f$.  
	\end{replemma}
	\begin{proof}
		If~$f'$ and~$f''$ are homotopy equivalences, we use \Cref{Buh3.1} to construct a diagram 
		\begin{center}
			\begin{tikzcd}
				a \arrow[r, "i_a",tail] \arrow[d,"f'"]\arrow[dr, phantom, "\square"]   & e \arrow[d,"g"]\arrow[r,two heads,"p_b"] & B\arrow[d,equal] \\
				c \arrow[r, "j",tail] \arrow[d,equal]   & z \arrow[d,"h"]\arrow[r,two heads,"q"]\arrow[dr, phantom, "\square"] & d\arrow[d,"f''"] \\
				c\arrow[r,tail,"i_c"]               & f \arrow[r,two heads,"p_d"] & d                                     
			\end{tikzcd}
		\end{center}
		with exact rows. Since pushouts and pullbacks of homotopy equivalences are homotopy equivalences, the maps~$g$ and~$h$ are homotopy equivalences, whence~$f=hg$ is.
		
		If~$f'$ and~$f''$ are cofibrations, then~$g$ is a cofibration since it is a pushout of~$f'$, and~$h$ is a cofibration as a direct consequence of \ref{Ex3}. Thus, the composite~$f=h\circ g$ is a cofibration.
		
		If~$f'$ and~$f''$ are fibrations, it is shown dually that~$f$ is a fibration.
	\end{proof}
	
	\begin{replemma}{lemma:3x3}
		Consider the commutative diagram with exact columns
		\begin{center}
			\begin{tikzcd}
				a' \arrow[d,tail,"i_a"]\arrow[r,"f'"] & e' \arrow[r,"g'"]\arrow[d,tail,"i_e"] & b'\arrow[d,tail,"i_b"] \\
				a \arrow[r,"f"]\arrow[d,two heads,"p_a"] & e \arrow[r,"g"]\arrow[d,two heads,"p_e"] & b\arrow[d,two heads,"p_b"] \\
				a''\arrow[r,"f''"]& e''\arrow[r,"g''"] & b''
			\end{tikzcd}
		\end{center}
		If the middle row and one of the other rows is exact, then the remaining row is exact.
	\end{replemma}
	\begin{proof}
		We assume that the top and middle rows are exact, and show that the bottom row is. The other case is dual.
		
		The map of exact sequences connecting the top and middle rows can be factored
		\begin{center}
			\begin{tikzcd}
				a' \arrow[r, "f'",tail] \arrow[d,"i_a",tail]\arrow[dr, phantom, "\square"]   & e' \arrow[d,"\overline{i_a}",tail]\arrow[r,two heads,"g'"] & b'\arrow[d,equal] \\
				a \arrow[r, "\overline{f}",tail] \arrow[d,equal]   & z \arrow[d,"\overline{i_b}",tail]\arrow[r,two heads,"\overline{g}"]\arrow[dr, phantom, "\square"] & b'\arrow[d,"i_b",tail] \\
				a \arrow[r,tail,"f"]               & e \arrow[r,two heads,"g"] & b                                    
			\end{tikzcd}
		\end{center}
		as asserted by \Cref{Buh3.1}. Our aim is to prove that~$f''$ is a cofibration and that the square
		\begin{center}
			\begin{tikzcd}
				a''\arrow[r,"f''"]\arrow[d,two heads] & e''\arrow[d,"g''"] \\
				0\arrow[r,tail] & b''
			\end{tikzcd}
		\end{center}
		is bicartesian. It suffices to show that the top square and outer rectangle in the diagram
		\begin{center}
			\begin{tikzcd}
				z\arrow[r,tail,"\overline{i_b}"]\arrow[d,"q",two heads] & e\arrow[d,"p_e",two heads] \\
				a''\arrow[r,"f''"]\arrow[d,two heads] & e''\arrow[d,"g''"] \\
				0\arrow[r,tail] & b''
			\end{tikzcd}
		\end{center}
		are bicartesian, where~$q$ is a cofiber of~$\overline{i_a}$ (by \Cref{Buh2.12}, the target of~$q$ can indeed be chosen to be~$a''$). Here, the outer rectangle is obtained by pasting two known bicartesian squares
		\begin{equation}\label{eq:3x3twosq}
		\begin{tikzcd}
		z\arrow[r,tail,"\overline{i_b}"]\arrow[d,"\overline{g}"]\arrow[dr, phantom, "\square"] & e\arrow[d,"g"] \\
		b'\arrow[r,"i_b"]\arrow[d,two heads]\arrow[dr, phantom, "\square"] & b\arrow[d,"p_b"] \\
		0\arrow[r,tail] & b''
		\end{tikzcd}
		\end{equation}
		and the upper square appears on the bottom right when applying \Cref{Buh3.1} to the map 
		\begin{center}
			\begin{tikzcd}
				a' \arrow[d, "i_a",tail] \arrow[r, "f'"] & e' \arrow[d, "i_e",tail] \\
				a \arrow[d, "p_a",two heads] \arrow[r, "f"]   & e \arrow[d, "p_e",two heads]  \\
				a'' \arrow[r, "f''"]                & e''                 
			\end{tikzcd}
		\end{center}
		of exact sequences. Thus, the top square and outer rectangle in \eqref{eq:3x3twosq} are bicartesian by construction, and the proof is complete.
	\end{proof}

	\bibliographystyle{plain}
	\bibliography{paper2}
	
\end{document}